\documentclass[11pt,a4paper]{article}
\usepackage{amsmath,amssymb,amsfonts,amsthm}
\usepackage{mathtools, thmtools}
\usepackage{graphicx}
\usepackage{tikz}
\usetikzlibrary{babel,cd,fit}
\usepackage{algorithm}
\usepackage{algpseudocode}
\usepackage{enumitem}
\usepackage{nicematrix}
\setlist{nosep}
\setlist[enumerate]{label=(\arabic*)}
\usepackage{etoolbox}
\usepackage{url}
\usepackage{multicol}
\usepackage{extpfeil}
\usepackage{thm-restate}
\usepackage{hyperref}
\usepackage[svgnames]{xcolor}
\numberwithin{equation}{section}
\usepackage[normalem]{ulem}
\newtheorem{theorem}{Theorem}[section]
\newtheorem{corollary}[theorem]{Corollary}
\newtheorem{proposition}[theorem]{Proposition}
\newtheorem{lemma}[theorem]{Lemma}

\theoremstyle{definition}
\newtheorem{definition}[theorem]{Definition}
\newtheorem{example}[theorem]{Example}
\newtheorem{construction}[theorem]{Construction}

\theoremstyle{remark}
\newtheorem{remark}[theorem]{Remark}
\newtheorem*{example*}{Example}

\definecolor{darkdgmcolor}{rgb}{0.0, 0.0, 0.8}
\definecolor{darkgreen}{rgb}{0.0, 0.8, 0.0}
\definecolor{purple}{RGB}{153,50,204}
\definecolor{dgmcolor}{RGB}{255,20,147}

\newcommand{\Z}{\mathbb{Z}}

\DeclareMathOperator{\lm}{lm}
\DeclareMathOperator{\lc}{lc}
\DeclareMathOperator{\lt}{lt}
\DeclareMathOperator{\coker}{coker}
\DeclareMathOperator{\im}{im}
\DeclareMathOperator{\Hom}{Hom}
\DeclareMathOperator{\Syz}{Syz}
\DeclareMathOperator{\lcm}{lcm}
\newcommand\Mon{\operatorname{\mathbf{Mon}}}
\DeclareMathOperator\supp{supp}
\DeclareMathOperator\mdeg{mdeg}
\newcommand\rank{\operatorname{rank}}
\newcommand{\rk}{\rank}
\newcommand\mon{(\mathrm{mon})}
\DeclareMathOperator\NF{NF}

\usepackage{etoolbox}
\pgfdeclarelayer{background}

\title{Relative Gr\"obner bases of modules and applications in persistence theory}

\author{
Fritz Grimpen\\
\small
Institute for Algebra, Geometry, Topology and their Applications (ALTA),\\
\small Department of Mathematics, University of Bremen, Germany\\
\small\texttt{grimpen@uni-bremen.de}
\and
Matthias Orth\\
\small Department of Mathematics,\\
\small KU Leuven, Belgium\\
\small \texttt{matthias.orth@kuleuven.be}
\and
Anastasios Stefanou*\\
\small Institute for Algebra, Geometry, Topology and their Applications (ALTA),\\
\small Department of Mathematics, University of Bremen, Germany\\
\small \texttt{stefanou@uni-bremen.de}
}

\date{}
\begin{document}

\maketitle

\enlargethispage{2\baselineskip}

{\let\thefootnote\relax
\footnotetext{*Corresponding author}
}

\begin{abstract}
    Finitely generated modules over the polynomial ring in $n$ indeterminates are isomorphic to quotients of finite rank free modules.
    We
     introduce
     a theory of relative Gröbner bases for those quotients of free modules and, equivalently, for pairs of submodules; we prove corresponding Buchberger- and Schreyer-type theorems.
    As applications of this theory, we consider three problems in persistence theory, which can be solved by relative Gröbner bases.
    First, we show that the relative Schreyer's theorem can be used to compute free presentations of complexes of finitely generated torsion-free modules.
    In contrast to previous approaches, this allows computation of free presentations for multicritical persistent homology directly at the chain module level without additional topological constructions.
    Second, any finitely generated Artinian module embeds in an Artinian injective hull, giving rise to a flat-injective presentation.
    We represent the embedding of the module in this injective hull by a quotient of a free module and apply the relative Schreyer's theorem to construct an algorithm for the computation of a free presentation from a flat-injective presentation.
    Third, we investigate how free presentations, and more generally free resolutions, obtained by the two preceding applications can be minimized by standard reduction techniques.

\end{abstract}

\tableofcontents

\section{Introduction}
In the theory of modules over the polynomial ring $R\coloneqq\Bbbk[X_1,\dotsc,X_n]$,
one cannot, in general, assign a  complete discrete invariant to each
isomorphism class of $R$-modules due to classical results of Carlsson and Zomorodian~\cite{carlsson2009theory}.
This is reflected by the fact that every finitely presented $\Z^n$-graded $R$-module admits a unique minimal free resolution
only up to noncanonical isomorphism. 
Thus, there is no canonical minimal resolution associated to each module isomorphism class.
Nevertheless, 
  a distinguished construction becomes available once a module is embedded
into a fixed free module and a monomial order on that free module is chosen.

Buchberger~\cite{Buchberger2006} showed that if $M \subseteq R^d$ is a submodule of a free module and a monomial
order on~$R^d$ is fixed, then there exists a unique
\emph{reduced Gröbner basis} of~$M$, independent of the choice of input
generators of~$M$.
Schreyer~\cite{Schreyer1980MastersThesis} subsequently showed that Gr\"obner bases allow one to compute kernels of morphisms between free modules, and in particular to explicitly determine a free presentation of~$M $ from any Gr\"obner basis~$G$ of~$M\subseteq R^d$.

\paragraph{Relative Gr\"obner bases.}
Hashemi, Orth, and Seiler \cite{HashemiOrthSeiler2021} developed the notion of \emph{relative Gröbner bases} for ideals~$J/I$ of the quotient ring~$R/I$ and more generally for $R/I$-submodules
  $V/I^d$ of $(R/I)^d$.
 For a fixed monomial order, they proved the existence of a unique reduced
relative Gröbner basis that is independent of the chosen generators of $I\subseteq R$ and $V\subseteq R^d$.
In particular, they showed a relative version of Schreyer's theorem for computing a free presentation of $J/I$ from a given Gr\"obner basis of $J$ relative to $I$.

In this work, we develop the analogous theory of Gr\"obner bases in the
setting of $R$-submodules $V/U$ of $R^d/U$ and obtain the analogous Schreyer's theorem for computing a free presentation of $V/U$.
 Before \cite{HashemiOrthSeiler2021}, other general frameworks of Gr\"obner bases over very broad classes of rings were developed by
Spear, Zacharias, and Mora \cite{mora2015zacharias_ISSAC,Mora2020Zacharias,Zacharias1978} and subsequent algorithmic refinements of Schreyer's theorem were introduced in \cite{EroecalEtAl2016Syzygies,LaScalaStillman1998}.
For an exposition of the theory of Gröbner bases for modules we refer to the books \cite{AdamsLoustaunau,CLOUsing,KreuzerRobbiano2}.

 Although these theories and algorithms provide 
tools for canonical forms and syzygy computations for $R/I$-submodules of $(R/I)^d$, respectively, they do not explicitly define notions of Gröbner bases or  presentations for $R$-submodules $V/U$ of $R^d/U$.
To the best of our knowledge, the constructions and results established here are novel.

\paragraph{Free presentations and free resolutions of multigraded modules}
\emph{Persistence theory} in topological data analysis (TDA) has led to extensive study of \emph{$n$-graded modules} (i.e., $\mathbb{Z}^n$-graded $R$-modules) \cite{carlsson2009theory}, particularly in connection with minimal free presentations and resolutions of homology modules of the form
\(
V/U:=\ker g \,/\, \operatorname{im} f
\)
arising from chain complexes of torsion-free $R$-modules
\(
L \xrightarrow{f} M \xrightarrow{g} N
\)
\cite{lesnick2022computing,LenzenDissertation}.
A central computational challenge in multiparameter TDA is the efficient computation of minimal presentations and minimal free resolutions of such homology modules \cite{botnan2023introduction}. 
Progress has been made in special cases, most notably for bigraded modules \cite{lesnick2022computing,bauerSoCG2023,ChacholskiScolamieroVaccarino2017,kerberRolle} and for $1$-critical complexes \cite{carlsson2010computing}.

A common strategy proceeds in two steps. First, one replaces the given complex of torsion-free modules by a quasi-isomorphic\footnote{Two chain complexes~$C, D$ are quasi-isomorphic if there exist chain maps~$f\colon C \to D$ and~$g\colon D \to C$ inducing isomorphisms~$f_\ast\colon H_\ast(C) \to H_\ast(D)$ and~$g_\ast\colon H_\ast(D) \to H_\ast(C)$.}
 complex of free modules, for instance, via the \emph{total complex}  \cite{weibel1994introduction,LenzenDissertation} or similar constructions \cite{ChacholskiScolamieroVaccarino2017}.
 Second, one computes a free resolution of the homology $V/U$ by resolving
$V := \ker g$ and $U := \operatorname{im} f$ via
standard methods, using Schreyer's theorem to obtain generators of $V$.
 This is possible because $U, V$ embed in the free module~$M$
after the quasi-isomorphic replacement.
 One then lifts the inclusion
$U \hookrightarrow V$ to a homomorphism between their free resolutions. Any such lift yields a free resolution of $V/U$.

While effective for computing homological invariants, this procedure is inherently non-canonical.
The first step is well-defined only up to quasi-isomorphism, and the second up to chain homotopy.
Consequently, both the resulting complex of free modules from the first step and the resolution of it from the second step require intermediary choices and are not uniquely determined by the original torsion-free complex.

Consequently, while algorithms for computing minimal free presentations and resolutions of homology modules of finitely generated torsion-free $n$-graded complexes are available, our approach makes the intermediate choices explicit and systematic at the chain-module level through the specification of a monomial order.
Beyond homology modules, the need for algorithms for computing normal forms of free presentations and resolutions of $n$-graded modules is even more pronounced.
New constructions of bigraded modules in TDA, such as \emph{filtered chain complexes} \cite{chacholski2023decomposing,ChacholskiGiuntiLandi2021,de2011dualities,meehan2019structural,usher2016persistent,MemoliZhou2023}, \emph{persistence Steenrod modules} \cite{lupo2022persistence}, and notably \emph{persistent cup-modules} \cite{memoli2024persistent,DeyRathod2024,yarmola2010persistence} encode additional structure that is invariant under cochain-homotopy, but not under a general quasi-isomorphism.

To sum up, these considerations motivate the development of 
 relative Gr\"{o}bner techniques for the 
explicit computation of resolutions of $n$-graded modules, in particular for modules  $V/U \subseteq F/U$, where $F$ is an $n$-graded free module.

\paragraph{Flat–injective presentations of multigraded modules}
 The idea of \emph{flat-injective presentations} was introduced by Helm and Miller \cite{HelmMiller2005} to facilitate the computation of local cohomology, and was utilized by Miller \cite{Miller2025} to describe $n$-graded modules by generators, cogenerators, and $\Bbbk$-linear relations between them.
Every Artinian
 module admits a minimal flat–injective presentation, unique up to
noncanonical isomorphism. In particular, one can compute injective hulls of $n$-graded modules as shown in \cite{HelmMiller2005,BrownDraganov2021}.
  
Lenzen subsequently developed an algorithm to compute a flat-injective
presentation of homology modules, by constructing a quasi-inverse chain map from a given chain complex to the $n$-shift of its Nakayama dual \cite{lenzen2024computing}.
 
Building on Miller’s theory \cite{Miller2025,HelmMiller2005}, Grimpen and Stefanou \cite{GrimpenStefanou2025} reinterpreted both the
construction and the minimization of flat–injective presentations directly at the module
level rather than at the derived level.
They constructed flat-injective presentation matrices from the multiplication maps,
characterized minimality algebraically, and  gave a reduction algorithm that,
for generators $f_1,\dotsc,f_s$ of an $R$-submoule $M$ of an Artinian injective $R$-module $E$, decides whether each $f_i$ belongs to the
$R$-span of the others by solving linear systems over~$\Bbbk$.

Actually, the starting point of the present work was the observation that any finite length $n$-graded submodule of an injective module can be expressed as a subquotient~$M = V/U \subseteq F/U$ of a free module~$F$ for an appropriate choice of~$U$.
  Thus, membership and minimization problems reduce to computations in such
quotients.
This motivated us to extend the theory of relative Gröbner bases to arbitrary quotient modules $R^d/U$, and consequently to multigraded quotients $F/U$. This extension provides a framework for computing normal forms for submodules $V/U$ of $R^d/U$, understood as unique remainders upon division of
 vectors of polynomials modulo $U$, and enables the effective treatment of fundamental computational problems, including membership testing, generator minimization, and the computation of presentations and resolutions of $V/U$.

Rather than recomputing Gröbner bases at each reduction step, we compute a single relative Gröbner basis and reuse it throughout the process. This improves computational efficiency.

\paragraph{Our contribution}
We develop a relative Gröbner basis theory for general $R$-modules, and study applications in persistence theory.

In Section \ref{sec:preliminaries} we recall the preliminaries.~
In Section \ref{sect:rel-gb}, we 
develop the analogous framework of \cite{HashemiOrthSeiler2021} from submodules $V/I^d$ of
$R^d/I^d$ to 
subquotients $V/U\subseteq R^d/U$.
Since every finitely generated module can be realized as such a quotient, this setting is
fully general, and thus can  be transferred to the multigraded module setting (replacing $R^d$ with a multigraded free module $F$).
Within this framework we establish existence, uniqueness, and syzygy constructions for
reduced relative Gröbner bases, leading to distinguished free presentations and resolutions for $V/U$, that are independent of the choice of generators of $V$ and of $U$.
 The following theorems summarize the main structural results.

\begin{restatable}[Relative Buchberger's theorem]{theoremalph}{TheoremA}%
    \label{thm:Uniqueness_And_Correspondence_Via_NF:intro}
    Let $U\subseteq V\subseteq R^d$ be submodules and  $\prec$ a fixed monomial order on $R^d$. Then $V$ has a unique reduced Gröbner basis relative to $U$.
\end{restatable}

\begin{restatable}[Relative Schreyer's theorem]{theoremalph}{TheoremB}%
    \label{thm:KernelOfPresentation:intro}
     Let $U \subseteq V \subseteq R^d$ be submodules. Given a (resp.~the unique reduced) Gr\"{o}bner basis $H$ of $V$ relative to $U$, there are  
     (uniquely) induced
       Gr\"{o}bner bases $H^{(i)}$, $i\geq 1$ of the syzygy $R$-modules $\Syz^{i}(H)\subseteq R^{t_{i}}$, where $t_i \coloneqq \lvert H^{(i-1)} \rvert$ and $H^{(0)}\coloneqq H$.
\end{restatable}

\noindent In Section \ref{sect:mgraded-pres} we study the following applications in persistence theory.

\paragraph{1.~Computing free resolutions of homology modules}

Theorems~\ref{thm:Uniqueness_And_Correspondence_Via_NF:intro} and \ref{thm:KernelOfPresentation:intro} provide a theoretical foundation for  computations with arbitrary submodules $V/U\subseteq R^d/U$.
 This is easier to see in the setting of $n$-graded modules, as any $n$-graded module $M$ is a submodule of a graded injective module $E$, whose $\mathbb N^n$-graded part~$E_{\mathbb N^n} = \bigoplus_{a \in \mathbb N^n} E_a$
 in turn can be expressed as a quotient $E_{\mathbb N^n} \cong F/U$ of a free $n$-graded module $F$ (see Rem.~\ref{rem:E as Rd mod U}).
 By applying Theorem~\ref{thm:KernelOfPresentation:intro}, we obtain the following result.

\begin{restatable}{theoremalph}{TFHomology}%
    \label{thm:homology presentation}
    Let  $L\xrightarrow{f}M\xrightarrow{g}N$ be a chain complex
    of finitely generated $n$-graded torsion-free modules (i.e.~$L,M,N$ are submodules of free modules $F_L,F_M,F_N$).
    A presentation of the homology of the complex can be computed using the relative Schreyer's theorem.

    In particular, for a given monomial order $\prec$ on $F_M$, the unique reduced Gr\"obner basis of $V:=\operatorname{ker}g$ relative $U:=\operatorname{im}f$ determines a presentation of the homology module $V/U$ that depends solely on $V$ and $U$, and in particular not on the choice of input generators of $V$ and $U$.
\end{restatable}

This allows computation of distinguished minimal free resolutions
of multigraded homology modules.

Theorem~\ref{thm:homology presentation} may be regarded as complementary to the presently available methods for computing a free-cofree presentation:
\begin{itemize}
\item The original method of Helm and Miller (see \cite{HelmMiller2005} for the mathematical foundations and \cite{HoferKahle2026} for a recent implementation) relies on the availability of minimal free presentations.
\item The method of Lenzen (see \cite[Cor.~5.1]{lenzen2024computing} for homology modules and~\cite[Thm.~6.1]{lenzen2024computing} for arbitrary modules) relies on the availability of minimal free resolutions.
\item In comparison to specialized algorithms for computing minimal free resolutions and presentations (see \cite{kerberRolle, lesnick2022computing} for Artinian graded~$k[X_1, X_2]$-modules and \cite{LaScalaStillman1998} for submodules of free~$k[X_1, \dotsc, X_n]$-modules) our result is completely general, as every module can be represented as a subquotient of a free module.
\end{itemize}

\paragraph{2.~Free resolutions from free-cofree presentations}
 Let $\varphi\colon F \to E$ be a free-cofree presentation that is represented by a free-cofree matrix.
 By bringing the free-cofree matrix of~$\varphi$ to its Gröbner form using Algorithm~\ref{alg:BuchbergerFlange} and identifying $\operatorname{im}\varphi\subseteq E$ with a subquotient~$V/U$ of a free module, see Remark~\ref{rem:E as Rd mod U}, we can utilize Theorem~\ref{thm:KernelOfPresentation:intro} to obtain a free presentation of the module~$V/U$.
Thus, we have the following result.

\begin{restatable}{theoremalph}{FlangeToFree}%
    \label{thm:free presentation of a free-cofree presentation}
    For any free-cofree matrix $A$ in Gröbner form, Algorithm~\ref{alg:FreePresentation} terminates and computes a free presentation of the module $\im A$.
\end{restatable}

The upshot of Theorem~\ref{thm:free presentation of a free-cofree presentation} is that given a free-cofree presentation of an $R$-module~$M$, we can compute a homomorphism
\[ \partial_1\colon R^s \longrightarrow R^l \]
of free $R$-modules, represented as a matrix w.r.t.\ the standard bases, satisfying~$\coker{\partial_1} \cong M$.
An implementation of Algorithm~\ref{alg:FreePresentation} is available as a Julia package, see \cite{software}.
Thus, after initial application of Theorem~\ref{thm:free presentation of a free-cofree presentation}, the problem of computing a free resolution of~$M$ is reduced to the free resolution problem for the submodule~$\ker{\partial_0}$ of the free $R$-module~$R^s$, see \cite{LaScalaStillman1998,EroecalEtAl2016Syzygies}.
 This enables computations of free resolutions for modules represented by free-cofree presentations.
Using the $n$-graded pruning, see Section~\ref{sect:minimal-res}, this leads to a method for computation of minimal free resolutions.

These results are complementary to the results of Helm and Miller \cite{HelmMiller2005} and Lenzen \cite[Cor.~5.1]{lenzen2024computing} as we provide algorithms to convert free-cofree presentations to free presentations first, and from those, using standard methods, to free resolutions.

Furthermore, with Theorem~\ref{thm:free presentation of a free-cofree presentation} it is possible to reduce the set of generators of a subquotient~$V/U \subseteq F/U$.
This provides an alternative to the reduction algorithm in \cite[Algorithm~4.8]{GrimpenStefanou2025} for free-cofree presentations, see Proposition~\ref{prop:reduced free cofree} and Corollary~\ref{cor:_reduced free cofree}.

\medskip
\noindent In Section \ref{sect:minimal-res} we give a pruning technique for minimizing $n$-graded free resolutions of submodules $V/U\subseteq F/U$, where $F$ is free.

\section{Preliminaries}
\label{sec:preliminaries}
Throughout this work, let $\Bbbk$ be a field.
Every ring is regarded as commutative, and for any (commutative) ring $R$ an \emph{$R$-module} is generally understood as an abelian group with a scalar multiplication compatible with the additive structure of the underlying ring $R$.
Particularly the free module with $d \in \mathbb N$ generators is denoted by $R^d$ and is endowed with the standard basis $(\mathbf e_1, \dotsc, \mathbf e_d)$.
As a convention, we will identify any concrete (graded) free-to-free homomorphism $\partial_i\colon F_i\to F_{i-1}$ with its (graded) matrix with respect to the standard bases. Given a module $M$, we denote any free cover $F_0\to M\to 0$ by $\varepsilon\colon F_0\to M\to 0$.

\subsection{Polynomial rings and Gröbner bases}
We recall the classical definitions and results from standard Gr\"obner basis theory.

We denote, for $n \in \mathbb N$, the \emph{polynomial ring in $n$ indeterminates with coefficients in $\Bbbk$} by $\Bbbk[X_1, \dotsc, X_n]$.
For an \emph{exponent vector} $\alpha = (\alpha_1, \dotsc, \alpha_n) \in \mathbb N^n$ we write $X^{\alpha}\coloneqq X_1^{\alpha_1} \dotsm X_n^{\alpha_n}$.
The elements of the form $X^{\alpha}$ for $\alpha \in \mathbb N^n$ are called \emph{monomials}, and every element $f \in \Bbbk[X_1, \dotsc, X_n]$ admits a finite expression
\[ f = \sum_{i = 1}^s c_i X^{\alpha_i} \]
for distinct $\alpha_1, \dotsc, \alpha_s \in \mathbb N^n$ and for $c_1, \dotsc, c_s \in \Bbbk$, which is unique up to reordering of the sum.

The \emph{join} of two lattice elements $\alpha, \beta \in \mathbb Z^n$, denoted by $\alpha \vee \beta$, is defined as the lattice element $(\max(\alpha_i, \beta_i))_{1 \leq i \leq n}$ for $\alpha = (\alpha_i)_i$ and $\beta = (\beta_i)_i$.
Dually, the \emph{meet} of $\alpha$ and $\beta$, denoted by $\alpha \wedge \beta$, is defined as the lattice element $(\min(\alpha_i, \beta_i))_{1 \leq i \leq n}$.
Notice that, if $\alpha, \beta$ are exponent vectors, then so are $\alpha \vee \beta$ and $\alpha \wedge \beta$.

In the following, let $R = \Bbbk[X_1, \dotsc, X_n]$.
For $d \in \mathbb N$, an element $\mathbf f \in R^d$ admits a finite expression
\begin{equation}
    \label{eqn:mon-decomp}
    \mathbf f = \sum_{i = 1}^d f_i \mathbf e_i
\end{equation}
for $f_1, \dotsc, f_d \in R$, where $\mathbf e_i$ denotes the $i$-th standard unit vector.
The elements of the form $X^{\alpha} \mathbf e_i$ in $R^d$ with $\alpha \in \mathbb N^n$ and $1 \leq i \leq d$ are called \emph{module monomials}.
We denote the set of module monomials in $R^d$ by $\Mon(R^d)$, and it is obvious that $\Mon(R^d)$ is a basis of the $\Bbbk$-vector space $R^d$.
Notice that in the case $d = 1$ a module monomial in $R^1 = R$ is simply a monomial.
Then, by decomposing $f_1, \dotsc, f_d$ in \eqref{eqn:mon-decomp} into monomials, it follows that $\mathbf f$ admits a unique, finite expression
\[ \mathbf f = \sum_{\mathbf m \in \Mon(R^d)} c_{\mathbf m} \mathbf m \]
as a linear combination of distinct module monomials with coefficients $c_{\mathbf m} \in \Bbbk$.

We say that a polynomial $\mathbf g \in R^d$ \emph{divides} a polynomial $\mathbf h \in R^d$ if there exists a polynomial $f \in R$ such that $\mathbf h = f \cdot \mathbf g$.
Generalizing this, we say that a finite set $G = \{ \mathbf g_1, \dotsc, \mathbf g_t \} \subseteq R^d$ \emph{divides} a polynomial $\mathbf h \in R^d$ if there exist $f_1, \dotsc, f_t \in R^d$ such that $\mathbf h = \sum_{j=1}^t f_j \cdot \mathbf g_j$.
The \emph{$R$-span}~$\langle G \rangle_R$ of a subset $G \subseteq R^d$ is the submodule of all finite linear combinations $\sum_{\mathbf g \in G} f_{\mathbf g} \mathbf g$ with $f_{\mathbf g} \in R$.

We call a total order~$\preceq$ on $\Mon(R^d)$ a \emph{monomial order} on $R^d$ if
\begin{enumerate}[label=(\roman*)]
    \item every subset $S \subseteq \Mon(R^d)$ admits a minimal element with respect to $\preceq$,
    \item for all $i \in \{ 1, \dotsc, n \}$ and $\alpha \in \mathbb N^n$, it holds $\mathbf e_i \preceq X^\alpha \mathbf e_i$, and
    \item for all $i, j \in \{ 1, \dotsc, n \}$ and $\alpha, \beta, \gamma \in \mathbb N^n$, it holds $X^{\alpha + \beta} \mathbf e_i \preceq X^{\alpha + \gamma} \mathbf e_j$ whenever $X^\beta \mathbf e_i \preceq X^\gamma \mathbf e_j$.
\end{enumerate}

Let $\preceq$ be a monomial order on $\Mon(R^d)$.
For a polynomial $\mathbf f = \sum_{j = 1}^t c_j \mathbf m_j \in R$ with $\mathbf m_1 \prec \dotsb \prec \mathbf m_t \in \Mon(R^d)$, $c_j \in \Bbbk$, and $c_t \neq 0$, we call the term $\lt(\mathbf f) = c_t \mathbf m_t$ the \emph{leading term} of $\mathbf f$.
The \emph{leading coefficient} of $\mathbf f$ is defined as $\lc(\mathbf f) = c_t$ and the \emph{leading monomial} of $\mathbf f$ as $\lm(\mathbf f) = \mathbf m_t$.
For a submodule $U \subseteq R^d$ we write $\lm(U)$ and $\lt(U)$ to denote the subspaces generated by $\{ \lm(\mathbf f) \mid \mathbf f \in U \}$ and $\{ \lt(\mathbf f) \mid \mathbf f \in U \}$, respectively.
Moreover, for a polynomial $\mathbf f$ with leading monomial $\lm(\mathbf f) = X^{\alpha} \mathbf e_j$, we call the exponent vector $\alpha \in \mathbb N^n$ the \emph{leading degree of $\mathbf f$}.
For a submodule $U \subseteq R^d$ we call the set $\Mon(R^d) \setminus \lm(U)$ the \emph{sous-escalier of $U$}.

\begin{definition}[Gröbner basis]
    Let $U \subseteq R^d$ be a submodule.
    \begin{enumerate}[label=(\arabic*)]
        \item A \emph{Gröbner basis}~$G$ of $U$ is a finite set $G = \{ \mathbf g_1, \dotsc, \mathbf g_t \}$ such that $\langle G \rangle_R = U$ and $\langle \lm(G) \rangle_R = \lm(U)$.
        \item A Gröbner basis~$G$ is \emph{minimal} if there exist no divisibility relations between the leading monomials of $G$.
        \item A minimal Gröbner basis~$G$ is \emph{reduced} if, for all $\mathbf g_i \in G$, $\lc(\mathbf g_i) = 1$ and $\mathbf g_i - \lm(\mathbf g_i)$ is a $\Bbbk$-linear combination of monomials not in $\lm(U)$.
    \end{enumerate}
\end{definition}

Given a monomial order $\preceq$, a submodule $U \subseteq R^d$ and a Gröbner basis $G$ of $U$, for an $\mathbf f \in R^d$ the remainder on polynomial division of $\mathbf f$ by $G$ is uniquely determined by the choice of monomial order.
This remainder is called the \emph{normal form for $\mathbf f$ modulo $U$} and is denoted by $\NF_\preceq(\mathbf f, U)$ or $\NF(\mathbf f, U)$.
We extend the definition of $\NF(\mathord-, U)$ pointwise to finite subsets $X \subseteq R^d$.
For more details, see \cite{cox1997ideals}.

The existence of Gröbner bases for arbitrary submodules of $R^d$ follows from Buchberger's criterion and the algorithm derived from it, see \cite[Chapter 5]{BeckerWeispfenning} and \cite[Chapter 2]{cox1997ideals} for instance.
The \emph{S-polynomial} of a pair $\mathbf f, \mathbf g$ of module polynomials in $R^d$ is
\[ \mathbf S(\mathbf f, \mathbf g) \coloneqq \frac{\lcm(\lm(\mathbf f), \lm(\mathbf g))}{\lt(\mathbf f)} \mathbf f - \frac{\lcm(\lm(\mathbf f), \lm(\mathbf g))}{\lt(\mathbf g)} \mathbf g \]
if $\lm(\mathbf f) = X^{\alpha} \mathbf e_k$ and $\lm(\mathbf g) = X^{\beta} \mathbf e_k$ for some $1 \leq k \leq d$, and $\mathbf S(\mathbf f, \mathbf g) = 0$ otherwise.

\begin{remark}
     A finite subset $F\subseteq U$ of a submodule $U\subseteq R^d$ is a Gr\"{o}bner basis of $U$ if and only if for all $\mathbf{f},\mathbf{g}\in F$, their S-polynomial
 admits a \emph{standard representation}, namely an expression as
 \[ \mathbf S(\mathbf{f},\mathbf{g})=\sum_{\mathbf{h}\in F} q_{\mathbf{h}} \mathbf{h} \text, \] with $q_{\mathbf{h}}\in R$ and $\lm(q_{\mathbf{h}}\mathbf{h})\preceq \lm(\mathbf S(\mathbf{f},\mathbf{g}))$ for all $\mathbf{h}\in H$.
\end{remark}

\begin{theorem}[Buchberger's criterion, see {\cite[Theorem 6, §6, Chapter 2]{cox1997ideals}}]
    Let $U \subseteq R^d$ be a submodule.
    A finite set $G = \{ \mathbf g_1, \dotsc, \mathbf g_t \}$ is Gröbner basis of $U$ if and only if the division algorithm reduces the $S$-polynomials $\mathbf S(\mathbf g_i, \mathbf g_j)$ to zero for all $1 \leq i < j \leq t$.
\end{theorem}

Further, every submodule $U \subseteq R^d$ has a unique reduced Gröbner basis, which does not depend on the input generators of $U$, and depends only on the choice of the monomial order $\preceq$ on $R^d$, cf.\ \cite{Buchberger2006,cox1997ideals} and \cite[Section 5.2]{CLOUsing}.
If the submodule $U$ is \emph{monomial}, which means that it is generated by module monomials, then the reduced Gröbner basis of $U$ does not depend on the choice of monomial order $\preceq$ at all, i.e.\ it is a \emph{universal} Gröbner basis. The universal Gröbner basis of a monomial submodule coincides with its minimal generating set.

However, the unique reduced Gröbner basis $G$ of a submodule $U \subseteq R^d$ is not necessarily a minimal set of generators of $U$; even the number of elements in $G$ can vary for different choices of monomial orders, see Example~\ref{ex:Resolution} and Section~\ref{ssect:gen-min}.

Next, we recall Schreyer’s theorem, which provides an effective Gröbner basis method for computing generating sets for the syzygy module, the module of relations among the generators of the module. In particular, given a Gröbner basis of a submodule of a free module, Schreyer’s construction yields a Gröbner basis for the syzygy module.

\begin{theorem}[Schreyer, cf.\ {\cite[pp.~222--225]{CLOUsing}}]\label{thm:SchreyerConstruction}
    Let $U \subseteq R^d$ be a submodule and $G = \{ \mathbf g_1, \dotsc, \mathbf g_s \}$ a Gröbner basis of $U$.
    For all $1 \leq i, j \leq s$ such that $\mathbf S(\mathbf g_i, \mathbf g_j) \neq 0$ choose expressions $\mathbf S(\mathbf g_i, \mathbf g_j) = \sum_{k=1}^s a_{ijk} \mathbf g_k$ where $a_{ijk} \in R$ and $\lt(a_{ijk} \mathbf g_k) \preceq \lt(\mathbf S(\mathbf g_i, \mathbf g_j))$, and define
    \[ \mathbf S_{ij} \coloneqq \frac{\lcm(\lm(\mathbf g_i), \lm(\mathbf g_j))}{\lt(\mathbf g_i)} \mathbf e_i - \frac{\lcm(\lm(\mathbf g_i), \lm(\mathbf g_j))}{\lt(\mathbf g_j)} \mathbf e_j - \sum_{k = 1}^s a_{ijk} \mathbf e_k \in R^s \text. \]
    The sequence
    \[ \begin{tikzcd}[ampersand replacement=\&]
        R^s \otimes_R R^s \rar["\partial_1"] \& R^s \rar["\varphi"] \& U \rar \& 0 \rlap{\text,}
    \end{tikzcd} \]
    where $\varphi(\mathbf e_i)\coloneqq \mathbf g_i$ and $\partial_1(\mathbf e_i \otimes \mathbf e_j)\coloneqq \mathbf S_{ij}$, is exact.

    Moreover, $\{ \mathbf S_{ij} \mid 1 \leq i < j \leq s \}$ forms a Gröbner basis of $\ker \varphi = \Syz(\mathbf g_1, \dotsc, \mathbf g_s)$ with respect to the monomial order $\prec_G$ defined by
    \[ X^\alpha \mathbf e_i \prec_G X^\beta \mathbf e_j \]
    if $\lt(X^\alpha \mathbf g_i) \prec \lt(X^\beta \mathbf g_j)$, or if $\lt(X^\alpha \mathbf g_i) = \lt(X^\beta \mathbf g_j)$ and $i > j$, for $\alpha, \beta \in \mathbb N^n$ and $1 \leq i, j \leq s$.
\end{theorem}

\begin{remark}
\label{rmk:Schreyer thm for obtaining kernel of morphism of free modules}
     As a consequence of Schreyer's theorem, kernels of homomorphisms between free modules can be computed algorithmically. Indeed, once a homomorphism of free $R$-modules is encoded effectively—for example, by a (monomial) matrix whose columns represent the images of the standard basis elements—its image is the $R$-submodule generated by these column vectors. Applying Schreyer’s theorem to this generating set produces a Gröbner basis for the corresponding syzygy module, which coincides with the kernel of the homomorphism.
\end{remark}

\subsection{Graded rings and modules}
Next we recall the basics of graded rings, where we focus solely on the classical case of $\mathbb Z^n$-graded rings, with special attention to the polynomial rings.

For further background on graded rings and their commutative algebra, we refer the reader to the textbook of Bruns and Herzog \cite{BrunsHerzog1998}.

\begin{definition}
    Let $n \in \mathbb N$.
    \begin{enumerate}
        \item An \emph{$n$-graded ring} $R$ is a ring $R$ with a decomposition $R = \bigoplus_{\alpha \in \mathbb Z^n} R_\alpha$ into additive subgroups $R_\alpha$ such that $R_\alpha R_\beta \subseteq R_{\alpha + \beta}$ for all $\alpha, \beta \in \mathbb Z^n$.

        \item An ideal $I$ of $R$ is said to be \emph{homogeneous} if $I = \bigoplus_{\alpha \in \mathbb Z^n} I \cap R_\alpha$.

        \item An \emph{$n$-graded $R$-module} $M$ is an $R$-module $M$ with a decomposition $M = \bigoplus_{\alpha \in \mathbb Z^n} M_\alpha$ into additive subgroups $M_\alpha$ such that $R_\alpha M_\beta \subseteq M_{\alpha + \beta}$ for all $\alpha, \beta \in \mathbb Z^n$.

        \item An $R$-homomorphism $f\colon M \to N$ is \emph{$n$-graded} if $f(M_\alpha) \subseteq N_\alpha$ for all $\alpha \in \mathbb Z^n$.
    \end{enumerate}
\end{definition}

Let $R = \bigoplus_{\alpha \in \mathbb Z^n} R_\alpha$ be an $n$-graded ring and $M = \bigoplus_{\alpha \in \mathbb Z^n} M_\alpha$ an $n$-graded $R$-module.
We call an element $m \in M$ that is contained in $M_\alpha$, for $\alpha \in \mathbb Z^n$, \emph{homogeneous} of \emph{degree $\alpha$}.
We say that an $n$-graded module $M = \bigoplus_{\alpha \in \mathbb Z^n} M_\alpha$ is \emph{finitely supported} if the set $\supp M\coloneqq\{ \alpha \in \mathbb Z^n \mid M_{\alpha} \neq 0 \}$, which is called the \emph{support of $M$}, is finite.

Grading on a ring is a proper structure and has to be regarded as such.
In particular, every ring can be endowed with some $n$-grading.

\begin{remark}%
    \label{rem:triv-grad}
    Every ring $R$ can be regarded as an $n$-graded ring with the \emph{trivial $n$-grading} that is given by $R = \bigoplus_{\alpha \in \mathbb Z^n} R_\alpha$ with $R_\alpha = R$ for $\alpha = 0$ and $R_\alpha = 0$ for $\alpha \neq 0$.
\end{remark}

The most important cases of graded rings are the polynomial rings in $n$ indeterminates with the \emph{standard gradings}.
For $n \in \mathbb N$ we consider the polynomial ring $R = \Bbbk[X_1, \dotsc, X_n]$ in $n$ indeterminates $X_1, \dotsc, X_n$.
For a term $f = \lambda X^{\alpha} \in R$, where $\alpha \in \mathbb N^n$ is an exponent vector, we say that $f$ has \emph{degree $\alpha$}.
Since the degree of any monomial is unique, we write $\mdeg f$ to denote the degree of $f \in \Mon(R)$.
More general, if $\mathbf f = \lambda X^{\alpha} \mathbf e_i \in F$ is a module monomial, for $F$ a free $R$-module, we say that $\mdeg \mathbf f = \alpha$.

The $n$-graded structure on $R = \Bbbk[X_1, \dotsc, X_n]$ is given by
\[ R = \bigoplus_{\alpha \in \mathbb Z^n} R_\alpha \]
where $R_\alpha$ is the $\Bbbk$-linear subspace of $R$ generated by monomials of degree $\alpha$.
In particular, $R_\alpha = 0$ for $\alpha \in \mathbb Z^n \setminus \mathbb N^n$.
 \begin{remark}
    An ideal $I\subseteq R$ is homogeneous with respect to the standard grading if and only if all its reduced Gr\"{o}bner bases consist of homogeneous polynomials~\cite[Theorem 8.3.2]{cox1997ideals}.
    This is due to the fact that $S$-polynomials and the division algorithm respect the grading.
\end{remark}

\begin{definition}
    An $n$-graded ring~$R$ is said to be \emph{(graded) local} if $R$ has a unique maximal homogeneous ideal.
\end{definition}

The polynomial ring $\Bbbk[X_1, \dotsc, X_n]$ with the standard grading is graded local with maximal ideal $\mathfrak m = (X_1, \dotsc, X_n)$, that is, $\mathfrak m$ consists of all monomials $f \in \Bbbk[X_1, \dotsc, X_n]$ with $f_0 = 0$.
The residue class field of $\Bbbk[X_1, \dotsc, X_n]$ is the base field~$\Bbbk$.

Let $M = \bigoplus_{\alpha \in \mathbb Z^n} M_\alpha$ be an $n$-graded $R$-module.
For $\beta \in \mathbb Z^n$ we define $\Sigma^\beta M = \bigoplus_{\alpha \in \mathbb Z^n} M_{\alpha - \beta}$, the \emph{$\beta$-shift of $M$}.
Notice that, if $\beta \leq \gamma$ in $\mathbb Z^n$, the multiplication by $X^{\gamma - \beta}$ induces a graded homomorphism $\Sigma^{\gamma} M \to \Sigma^\beta M$.

\begin{definition}[Graded hom and tensor product]
    Let $R$ be a $n$-graded ring and $M, N$ be graded $R$-modules.
    \begin{enumerate}
        \item The \emph{graded hom from $M$ to $N$} is the graded $R$-module
        \[ \hom_R(M, N) \coloneqq \bigoplus_{\alpha \in \mathbb Z^n} \Hom_R(\Sigma^\alpha M, N) \text. \]
        \item The \emph{graded tensor product of $M$ and $N$} is the graded $R$-module $M \otimes_R N$ with
        \[ (M \otimes_R N)_\alpha \coloneqq (\bigoplus_{\beta \in \mathbb Z^n} M_\beta \otimes_{\mathbb Z} N_{\alpha - \beta})/S_\alpha \text, \qquad \text{for $\alpha \in \mathbb Z$,} \]
        where $S_\alpha$ is the subgroup generated by all elements of the form
        \[ rm \otimes n - m \otimes rn \qquad \text{with $r \in R_\gamma$, $m \in M_{\gamma'}$ and $n \in N_{\gamma''}$} \]
        for all $\gamma, \gamma', \gamma'' \in \mathbb Z^n$ such that $\gamma + \gamma' + \gamma'' = \alpha$.
    \end{enumerate}
\end{definition}

We say that a graded $R$-module $M$ is \emph{projective} if the covariant graded hom functor $\hom_R(M, \mathord-)$ is exact.
Dually, $M$ is said to be \emph{injective} if the contravariant graded hom functor $\hom_R(\mathord-, M)$ is exact.
A left $R$-module $M$ with exact functor $\mathord- \otimes_R M$ is said to be \emph{flat}.

By a classical result for local rings, any finitely generated projective graded $R$-module is \emph{graded free}, which means that for such module $M$ there exists a family $(\beta_j)_{j \in J}$ in $\mathbb Z^n$ such that $M \cong \bigoplus_{j \in J} \Sigma^{\beta_j} R$, cf.\ \cite[Theorem 19.19, p.~292]{Lam2001} and \cite{HoeppnerLenzing1981}.
As an immediate consequence, homomorphisms between finitely generated projective graded modules can be represented by matrices.

The standard technique to construct flat modules is by localization.

\begin{definition}[Graded localization]
    Let $R$ be a graded ring.
    \begin{enumerate}
        \item A \emph{graded multiplicative system} $S$ is a subset $S \subseteq R$ that consists of homogeneous elements such that $s s' \in S$ for all $s, s' \in S$.
        \item The \emph{graded localization} of $R$ by a graded multiplicative system $S$ is the set $S^{-1}R$ of formal fractions $r/s$ with $r \in R$ and $s \in S$.
        \item For a homogeneous prime ideal $\mathfrak p$ in $R$ the \emph{graded localization along $\mathfrak p$} is the set $R_{\mathfrak p} \coloneqq \{ s \in R \mid s \notin \mathfrak p\text,\ \text{$s$ homogeneous} \}^{-1}R$.
    \end{enumerate}
\end{definition}

For any graded multiplicative system $S$ the set $S^{-1}R$ has a natural ring structure, which is given for $r, r' \in R$, $s, s' \in S$ by $r/s + r'/s' = (r s' + r' s)/(ss')$ and $r/s \cdot r'/s' = (rr')/(ss')$.
Moreover, $S^{-1}R$ is naturally $n$-graded where an element $r/s \in S^{-1}R$, for $r \in R$ and $s \in S$ homogeneous of degrees $\alpha$ and $\beta$, respectively, has degree $\alpha - \beta$.

An \emph{injective hull}~$E$ of a (graded) module $M$ is a (graded) injective module~$E$ with a monomorphism $\varphi\colon M \to E$ such that every endomorphism $f\colon E \to E$ with $f \circ \varphi = \varphi$ is an automorphism.
It is well-known that every (graded) module~$M$ admits an injective hull, denoted by $E(M)$, which is uniquely determined up to isomorphism.

\subsection{Flat-injective presentations}

Let $R$ be a local graded ring with maximal ideal $\mathfrak m$ and residue class field $\Bbbk = R/\mathfrak m$.
For a graded $R$-module $M$ we denote by $M^\vee$ the \emph{Matlis dual of $M$} defined as $\hom_R(M, E(\Bbbk))$ where $E(\Bbbk)$ denotes the graded injective hull of the $R$-module $\Bbbk$.
Then we have an isomorphism
\begin{equation}%
    \label{eq:matlis-iso}
    M^\vee = \hom_R(M, E(\Bbbk)) \cong \bigoplus_{\alpha \in \mathbb Z^n} \Hom_{\Bbbk}(M_{-\alpha}, \Bbbk) \text.
\end{equation}

\begin{definition}%
    \label{def:flange}
    Let $M$ be a graded $R$-module.
    \begin{enumerate}
        \item A \emph{flat-injective presentation} $\varphi$ for $M$ is an $R$-homomorphism $\varphi$ from a flat $R$-module $F$ into an injective $R$-module $E$ such that $\im \varphi \cong M$.
        \item The \emph{Matlis-dual presentation} of a flat-injective presentation $\varphi\colon F \to E$ is the $R$-homomorphism $\varphi^{\vee}\coloneqq \hom_R(\varphi, E(\Bbbk))\colon E^{\vee} \to F^{\vee}$, that is, $\varphi^{\vee}_\alpha$ is given by $\varphi^{\vee}_\alpha(f)\coloneqq f \circ \varphi$, for $\Hom_{\Bbbk}(E_{-\alpha}, \Bbbk)$, under the isomorphism \eqref{eq:matlis-iso}.
        \item A flat-injective presentation $\varphi$ is \emph{generator-minimal} if $\ker \varphi \subseteq \mathfrak mF$.
        \item A flat-injective presentation $\varphi$ is \emph{cogenerator-minimal} if $\varphi^{\vee}$ is generator-minimal.
    \end{enumerate}
\end{definition}

Flat-injective presentations were introduced by Miller \cite{Miller2025} under the term \emph{flange presentations}.
If a flat-injective presentation $\varphi\colon F \to E$ has free $F$ then we call $\varphi$ a \emph{free-injective presentation}.
If in addition $E$ is of the form $\bigoplus_{i \in I} \Sigma^{\alpha_i} E(\Bbbk)$, then we call $\varphi$ a \emph{free-cofree presentation}.

It is a direct consequence of the definition that a flat-injective presentation $\varphi\colon F \to E$ is generator-minimal if and only if $F$ is a flat cover of $\im \varphi$.
Further, it is due to Hyry and Puuska \cite{HyryPuuska2025} that such flat-injective presentation is cogenerator-minimal if and only if $E$ is an injective hull of $\im \varphi$.

A crucial property of flat-injective presentations is that they represent any module as a submodule of an injective module.
This can be used to perform computations on a module as long as the analytic structure of injective modules is well-understood.
In fact, for any injective module $E$ there exists a family $(\alpha_i)_{i \in I}$ in $\mathbb Z^n$ and a family $(\mathfrak p_i)_{i \in I}$ of homogeneous prime ideals in $R$ such that
\[ E \cong \bigoplus_{i \in I} \Sigma^{\alpha_i} E_R(R/\mathfrak p_i) \text, \]
where $E_R(R/\mathfrak p_i)$ denotes the injective hull of $R/\mathfrak p_i$, cf. \cite{Matlis1958, GotoWatanabe1978b}.
Thus, in order to obtain the analytic structure of an element $m \in \im \varphi$ it is sufficient to analyze the structures of $E(R/\mathfrak p)$ for homogeneous prime ideals $\mathfrak p$.

For the polynomial ring $\Bbbk[X_1, \dotsc, X_n]$ the injective hulls $E(R/\mathfrak p)$ correspond to factor modules of $E(R)$.
A concrete representation of $E(R)$ can be chosen as the Laurent polynomial ring $\Bbbk[X_1^{\pm1}, \dotsc, X_n^{\pm1}]$ regarded as a graded $\Bbbk[X_1, \dotsc, X_n]$-module.
For a prime ideal $\mathfrak p = (X_{k_1}, \dotsc, X_{k_t})$, with $1 \leq k_1 < \dotsb < k_t \leq n$, the injective hull of $R/\mathfrak p$ is given by
\[ E(R/\mathfrak p) \cong \Bbbk[X_1^{\pm1}, \dotsc, X_n^{\pm1}] / S_{\mathfrak p} \]
where $S_{\mathfrak p}$ is an appropriate submodule of $\Bbbk[X_1^{\pm1}, \dotsc, X_n^{\pm_1}]$ depending on $\mathfrak p$.
In particular, it holds $E(\Bbbk) = \Bbbk[X_1^{-1}, \dotsc, X_n^{-1}]$, which consists of Laurent polynomials with non-positive exponents.
See Section~\ref{sect:gb-free-cofree} and in particular Proposition~\ref{prop:inj-pos-graded} for a deduction of this representation using the theory of Čech complexes.

This leads to the idea that flat-injective presentations can be represented by matrices of special types.
For this purpose, we restrict our attention to the case where $\varphi\colon F \to E$ is a flat-injective presentation with $E \cong E(\Bbbk)^m$ for some $m \in \mathbb N$.
Let $\alpha_1, \dotsc, \alpha_s \in \mathbb Z^n$ and $\beta_1, \dotsc, \beta_t \in \mathbb Z^n$.

\begin{definition}
    A \emph{free-injective matrix} is a triple $(A, (\alpha_i)_i, (\beta_j)_j)$ where $A = (a_{ij})_{i,j}$ is an $s \times t$-matrix with entries in $\Bbbk$ such that $a_{ij} = 0$ if not $\beta_j \leq \alpha_i$, for $1 \leq i \leq s$ and $1 \leq j \leq t$.

    The degrees $\alpha_1, \dotsc, \alpha_s$ are called \emph{cogenerator degrees} and the degrees $\beta_1, \dotsc, \beta_t$ are called \emph{generator degrees}.
\end{definition}

Let $\alpha_1, \dotsc, \alpha_s, \beta_1, \dotsc, \beta_t \in \mathbb Z^n$, $F\coloneqq \bigoplus_{j = 1}^t \Sigma^{\beta_j} R$, and $E \coloneqq \bigoplus_{i = 1}^s \Sigma^{\alpha_i} E(\Bbbk)$.

\begin{proposition}
    Let $\alpha_1, \dotsc, \alpha_s, \beta_1, \dotsc, \beta_t \in \mathbb Z^n$.
    There is a one-to-one correspondence between $\Hom_R(F, E)$ and the set of graded matrices w.r.t.\ generator degrees $(\alpha_1, \dotsc, \alpha_s)$ and cogenerator degrees $(\beta_1, \dotsc, \beta_t)$.

    \begin{proof}
        Proofs of this result are given in \cite[pp.\ 385--386]{HelmMiller2005} and \cite[Proposition 5.18]{Miller2025}.
    \end{proof}
\end{proposition}

\begin{remark}%
    \label{rem:fg-fs-Artinian}
    For a finitely generated graded $\Bbbk[X_1, \dotsc, X_n]$-module $M$ the following conditions are equivalent, see also \cite[Proposition~18.1.2]{ChristensenFoxbyHolm2024}:
    \begin{itemize}
        \item $M$ is Artinian;
        \item $M$ can be embedded in $\bigoplus_{i = 1}^s \Sigma^{\alpha_i} E(\Bbbk)$ for some $\alpha_1, \dotsc, \alpha_s \in \mathbb Z^n$;
        \item $M$ is finitely supported, which means that $\{ a \in \mathbb Z^n \mid M_a \neq 0 \}$ is a finite set.
    \end{itemize}
    Therefore, every finitely generated Artinian $\Bbbk[X_1, \dotsc, X_n]$-module admits a flat-injective presentation $\varphi\colon F \to E$ where $F$ is free and $E$ is a finite direct sum of shifts of $E(\Bbbk)$.
    Such a flat-injective presentation admits a free-injective matrix, and the columns of such free-injective matrix correspond to analytic elements of the presented module.
\end{remark}

\begin{example}%
    \label{ex:flange-1}
    We consider the finitely generated graded $\Bbbk[X_1, X_2]$-module $M$ represented by the diagram
    \[ \begin{tikzcd}[ampersand replacement=\&, labels={outer sep=1ex}]
        \Bbbk \& \Bbbk^2 \& \Bbbk \\
        0 \& \Bbbk \& \Bbbk \rlap{\text,}
        \arrow[from=1-1, to=1-2, "{\left(\begin{smallmatrix} 0 \\ 1 \end{smallmatrix}\right)}"{fill=LightBlue}]
        \arrow[from=2-2, to=1-2, "{\left(\begin{smallmatrix} 1 \\ 0 \end{smallmatrix}\right)}"{fill=LightSalmon}]
        \arrow[from=1-2, to=1-3, "{\left(\begin{smallmatrix} 1 & 1 \end{smallmatrix}\right)}"{fill=LightCoral}]
        \arrow[from=2-1, to=1-1]
        \arrow[from=2-1, to=2-2]
        \arrow[from=2-2, to=2-3, "1"{fill=LightPink}]
        \arrow[from=2-3, to=1-3, "1"'{fill=LightGoldenrod}]
    \end{tikzcd} \]
    where the lower left cell denotes the degree $(0, 0)$ and the upper right cell denotes the degree $(2, 1)$.
    This module arises naturally as a lattice embedding of a non-trivial indecomposable representation of the Dynkin quiver~$D_4$, see \cite[Section 5.2]{lenzen2024computing} for the module, \cite[Definition 7]{buchet2022realizations} for a related class of indecomposable modules and \cite[Chapter 9]{Barot2015} for an account on indecomposable quiver representations.
    A realization of the module $M$ as a bifiltration $X$ with $H_1(X) \cong M$ is given in Figure~\ref{fig:flange-bifilt}.

    By applying the techniques from \cite[Theorem 4.3]{GrimpenStefanou2025}, this module can be represented by the free-injective matrix
    \[ A = (a_{ij})_{i, j} \coloneqq \begin{tikzpicture}[ampersand replacement=\&, baseline=(m-4-1.base)]
        \pgfsetlayers{background,main}
    	\matrix (m) [matrix of math nodes] {
    		\&[1em] (1,0) \& (0,1) \& (2,0) \& (1,1) \& (1,1) \& (2,1) \\
    		(1,0) \& 1 \& 0 \& 0 \& 0 \& 0 \& 0 \\
    		(0,1) \& 0 \& 1 \& 0 \& 0 \& 0 \& 0 \\
    		(2,0) \& 1 \& 0 \& 1 \& 0 \& 0 \& 0 \\
    		(1,1) \& 1 \& 0 \& 0 \& 1 \& 0 \& 0 \\
    		(1,1) \& 0 \& 1 \& 0 \& 0 \& 1 \& 0 \\
    		(2,1) \& 1 \& 1 \& 1 \& 1 \& 1 \& 1 \\
    	};
    	\node[rectangle, left delimiter={[}, right delimiter={]}, fit=(m-2-2.north west) (m-7-7.south east)] {};
        \begin{pgfonlayer}{background}
            \fill[LightPink] (m-4-2.north west) rectangle (m-4-2.south east);
            \fill[LightSalmon] (m-5-2.north west) rectangle (m-6-2.south east);
            \fill[LightBlue] (m-5-3.north west) rectangle (m-6-3.south east);
            \fill[LightCoral] (m-7-5.north west) rectangle (m-7-6.south east);
            \fill[LightGoldenrod] (m-7-4.north west) rectangle (m-7-4.south east);
        \end{pgfonlayer}
    \end{tikzpicture} \text, \]
    where the column and row labels denote the generator and cogenerator degrees, respectively.
    That is, the module~$M$ is isomorphic to the submodule of
    \[ E \coloneqq \Sigma^{(1, 0)} E(\Bbbk) \oplus \Sigma^{(0, 1)} E(\Bbbk) \oplus \Sigma^{(2, 0)} E(\Bbbk) \oplus \Sigma^{(1, 1)} E(\Bbbk)^2 \oplus \Sigma^{(2, 1)} E(\Bbbk) \]
    generated by the images of the elements
    \[ \mathbf g_j \coloneqq \sum_{i=1}^6 a_{ij} \mathbf e_i \in F \qquad \text{for $j \in \{ 1, \dotsc, 6 \}$.} \]

    A minimal flat-injective presentation for the module $M$ is given by
    \[ \tilde A = (\tilde a_{ij})_{i, j} \coloneqq \begin{tikzpicture}[ampersand replacement=\&, baseline=(m-3-1.north)]
    	\matrix (m) [matrix of math nodes] {
    		\&[1em] (1,0) \& (0,1) \\
    		(1,1) \& 1 \& 0 \\
    		(2,1) \& 1 \& 1 \\
    	};
    	\node[rectangle, left delimiter={[}, right delimiter={]}, fit=(m-2-2.north west) (m-3-3.south east)] {};
    \end{tikzpicture} \text. \]
    The minimal matrix~$\tilde A$ can be obtained by the minimization technique  in \cite[Algorithm 4.8]{GrimpenStefanou2025}.
    Topologically, the two columns correspond to the cycles $z_1$ and $z_2$ depicted in Figure~\ref{fig:flange-bifilt}.
    The death in degree $(1, 1)$, represented by the first column, corresponds to the fact that $z_1 - z_2$ is a boundary outside of the sublattice $\{ \alpha \in \mathbb Z^2 \mid \alpha \leq (1, 1) \}$.
    Analogously, $z_1$ and $z_2$ are boundaries outside of the sublattice $\{ \alpha \in \mathbb Z^2 \mid \alpha \leq (2, 1) \}$, which corresponds to the death in degree $(2, 1)$, represented by the second column.

    We use the free-injective presentations~$A$ and~$\tilde A$ as a rolling example throughout Section~\ref{sect:mgraded-pres} and Section~\ref{sect:minimal-res}.
    \begin{figure}
        \centering
        \begin{tikzpicture}[ampersand replacement=\&, box/.style={text width=25mm, draw, align=center}, simplex0/.style={radius=2pt, draw=black, fill=white, line width=1pt}, simplex1/.style={draw=black}, simplex2/.style={fill=lightgray}]
        	\matrix (m) [column sep=0.5cm, row sep=0.5cm] {
        		\node[box] (m-1-1) {
        			\tikz {
        				\path
        					(0, 0) coordinate (a) node[left=2pt] {$v_3$}
        					(-0.7, 1) coordinate (b) node [left=5pt, above] {$v_4$}
        					(0.7, 1) coordinate (c) node [right=5pt, below] {$v_2$}
        					(0, 2) coordinate (d) node [right=2pt] {$v_1$}
        					(0, 1) coordinate (z1)
        					(0, 1.35) coordinate (z2);
        				\node at (z1) {$z_1$};
        				\draw[->] (z1) +(2.2mm, 1.35mm) arc[start angle=30, delta angle=300, radius=2.7mm];
        				\path[simplex1]
        					foreach \x/\y in {a/b, a/c, b/d, c/d} {(\x) edge (\y)};
        				\path[simplex0]
        					foreach \x in {a, b, c, d} {(\x) circle {}};
        			}
        		};
        		\&
        		\node[box] (m-1-2) {
        			\tikz {
        				\path
        					(0, 0) coordinate (a) node[left=2pt] {$v_3$}
        					(-0.7, 1) coordinate (b) node [left=5pt, above] {$v_4$}
        					(0.7, 1) coordinate (c) node [right=5pt, below] {$v_2$}
        					(0, 2) coordinate (d) node [right=2pt] {$v_1$}
        					(0, 1) coordinate (z1)
        					(0, 1.35) coordinate (z2);
        				\path[simplex1]
        					foreach \x/\y in {a/b, a/c, b/c, b/d, c/d} {(\x) edge (\y)};
        				\path[simplex0]
        					foreach \x in {a, b, c, d} {(\x) circle {}};
        			}
        		};
        		\&
        		\node[box] (m-1-3) {
        			\tikz {
        				\path
        					(0, 0) coordinate (a) node[left=2pt] {$v_3$}
        					(-0.7, 1) coordinate (b) node [left=5pt, above] {$v_4$}
        					(0.7, 1) coordinate (c) node [right=5pt, below] {$v_2$}
        					(0, 2) coordinate (d) node [right=2pt] {$v_1$}
        					(0, 1) coordinate (z1)
        					(0, 1.35) coordinate (z2);
        				\path[simplex2]
        					foreach \x/\y/\z in {a/b/c} {(\x) -- (\y) -- (\z)};
        				\path[simplex1]
        					foreach \x/\y in {a/b, a/c, b/c, b/d, c/d} {(\x) edge (\y)};
        				\path[simplex0]
        					foreach \x in {a, b, c, d} {(\x) circle {}};
        			}
        		};
        		\\
        		\node[box] (m-2-1) {
        			\tikz {
        				\path
        					(0, 0) coordinate (a) node[left=2pt] {$v_3$}
        					(-0.7, 1) coordinate (b) node [left=5pt, above] {$v_4$}
        					(0.7, 1) coordinate (c) node [right=5pt, below] {$v_2$}
        					(0, 2) coordinate (d) node [right=2pt] {$v_1$}
        					(0, 1) coordinate (z1)
        					(0, 1.35) coordinate (z2);
        				\path[simplex1]
        					foreach \x/\y in {b/d, c/d} {(\x) edge (\y)};
        				\path[simplex0]
        					foreach \x in {a, b, c, d} {(\x) circle {}};
        			}
        		};
        		\&
        		\node[box] (m-2-2) {
        			\tikz {
        				\path
        					(0, 0) coordinate (a) node[left=2pt] {$v_3$}
        					(-0.7, 1) coordinate (b) node [left=5pt, above] {$v_4$}
        					(0.7, 1) coordinate (c) node [right=5pt, below] {$v_2$}
        					(0, 2) coordinate (d) node [right=2pt] {$v_1$}
        					(0, 1) coordinate (z1)
        					(0, 1.35) coordinate (z2);
        				\node at (z2) {$z_2$};
        				\draw[->] (z2) +(2.2mm, 1.35mm) arc[start angle=30, delta angle=300, radius=2.7mm];
        				\path[simplex1]
        					foreach \x/\y in {a/c, b/c, b/d, c/d} {(\x) edge (\y)};
        				\path[simplex0]
        					foreach \x in {a, b, c, d} {(\x) circle {}};
        			}
        		};
        		\&
        		\node[box] (m-2-3) {
        			\tikz {
        				\path
        					(0, 0) coordinate (a) node[left=2pt] {$v_3$}
        					(-0.7, 1) coordinate (b) node [left=5pt, above] {$v_4$}
        					(0.7, 1) coordinate (c) node [right=5pt, below] {$v_2$}
        					(0, 2) coordinate (d) node [right=2pt] {$v_1$}
        					(0, 1) coordinate (z1)
        					(0, 1.35) coordinate (z2);
        				\path[simplex1]
        					foreach \x/\y in {a/c, b/c, b/d, c/d} {(\x) edge (\y)};
        				\path[simplex0]
        					foreach \x in {a, b, c, d} {(\x) circle {}};
        			}
        		};
        		\\
        	};
        	\path[every node/.append style={anchor=north west, font=\small, color=gray}]
        		(m-1-1.north west) node {$(0, 1)$}
        		(m-1-2.north west) node {$(1, 1)$}
        		(m-1-3.north west) node {$(2, 1)$}
        		(m-2-1.north west) node {$(0, 0)$}
        		(m-2-2.north west) node {$(1, 0)$}
        		(m-2-3.north west) node {$(2, 0)$};
        	\path[every edge node/.style={}, every edge/.style={draw=none, edge node=node [every edge node] {$\subseteq$}}]
        		(m-1-1) edge (m-1-2)
        		(m-1-2) edge (m-1-3)
        		(m-2-1) edge (m-2-2)
        		(m-2-1) edge[every edge node/.style={rotate=90}] (m-1-1)
        		(m-2-2) edge (m-2-3)
        		(m-2-2) edge[every edge node/.style={rotate=90}] (m-1-2)
        		(m-2-3) edge[every edge node/.style={rotate=90}] (m-1-3);
        \end{tikzpicture}
        \caption{Bifiltration of simplicial complexes realizing the module $M$ from Example~\ref{ex:flange-1}.
        The bifiltration extends upwards and towards the right by filled tetrahedra, which are contractible.}
        \label{fig:flange-bifilt}
    \end{figure}
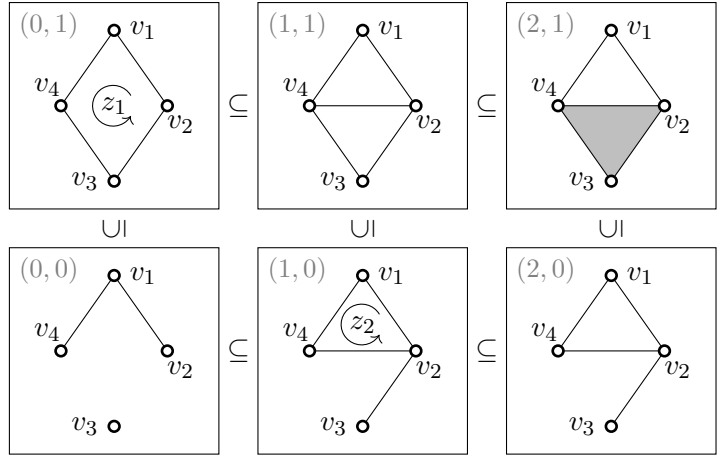
\end{example}

\section{Relative Gr\"obner basis theory}
\label{sect:rel-gb}

Throughout this section, let $\Bbbk$ be a field, and let $R\coloneqq \Bbbk[X_1, \dotsc, X_n]$ be the (ungraded) polynomial ring over $\Bbbk$ in $n$ indeterminates. We develop the theory of relative Gr\"{o}bner bases for subquotients of free modules $R^d$ by generalizing \cite[Section 3]{HashemiOrthSeiler2021}.

\subsection{Relative Gr\"obner bases}

In this section, we generalize the notion of Gröbner bases of submodules $V \subseteq R^d$ to pairs of submodules $U \subseteq V \subseteq R^d$, where $U$ acts as an annihilating submodule, meaning that we regard elements of $V$ modulo $U$, or equivalently, as their images in the quotient module $V / U$.
In particular, we define \emph{relative Gröbner bases} and state and prove a Buchberger-type criterion for relative Gröbner bases.

It is well-known under the term \emph{correspondence theorem} (also known as \emph{the fourth isomorphism theorem on modules}) that there is a one-to-one correspondence between the submodules of $R^d/U$ and the submodules $V \subseteq R^d$ satisfying $U \subseteq V$, see \cite{Jacobson1976} for instance.
In the following, we shall begin with the perspective from quotient modules $V / U$.
For each submodule $\bar{V}\subseteq R^d/U$, we write $\pi_U^{-1}(\bar{V})$ for the unique submodule $V$ of $R^d$ that contains $U$ and such that $\pi_U(V)=\bar{V}$, where $\pi_U$ is the canonical projection of the quotient $R^d/U$.

Before defining Gr\"{o}bner bases in this setting, we first need to clarify what we mean by leading monomials and coefficients in a quotient module $R^d/U$.
The definition is essentially the same as in the absolute case, with the vector space basis $\Mon(R^d)$ replaced by the sous-escalier of $U$ with respect to a given module monomial ordering.

\begin{definition}\label{def:LeadingMon_InQuotient}
    Let $U \subseteq R^d$ be a submodule and $\preceq$ a monomial ordering on $R^d$.
    \begin{enumerate}
        \item The set of \emph{monomials in $R^d/U$} is
        \[ \Mon_\prec(R^d/U) \coloneqq \Mon(R^d/U) = \{ \mathbf m + U \mid \mathbf m \in \Mon(R^d) \setminus \lm(U) \} \text. \]
        An element of $\Mon_\prec(R^d/U)$ is called a \emph{quotient monomial}.
        \item For non-zero $\mathbf f + U \in R^d/U$ the \emph{leading monomial} $\lm(\mathbf f + U)$ is defined as the leading term of the uniquely determined normal form for $\mathbf f$ modulo $U$.
        The \emph{leading coefficient} of $\lm(\mathbf f + U)$ is defined as the leading coefficient of $\NF(\mathbf f, U)$.
        \item For a submodule $\bar V \subseteq R^d/U$ define $\lm(\bar V)$ as the $R$-span of the leading monomials $\lm(\mathbf f + U)$ for all non-zero $\mathbf f + U \in \bar V$.
    \end{enumerate}
\end{definition}

This notion of monomial ordering directly gives rise to a definition of relative Gröbner bases for quotient modules $V / U$, where $U \subseteq V \subseteq R^d$ is a pair of submodules.

\begin{definition}\label{def:GB_ForNestedSubmodules}
    Let $U \subseteq R^d$ be a submodule and $\preceq$ a monomial ordering on $R^d$.

    \begin{enumerate}
        \item A finite generating set $H = \{ \mathbf h_1 + U, \dotsc, \mathbf h_t + U \}$ of $\bar  V$ is called a \emph{Gröbner basis} of $\bar V$ if for each module monomial $\mathbf m \in \lm(\bar V)$ there is an element $\mathbf h_i + U \in H$ such that $\lm(\mathbf h_i + U)$ divides $\mathbf m$.
        \item A Gröbner basis $H$ is \emph{minimal} if there exist no divisibility relations between the leading monomials of $H$.
        \item A minimal Gröbner basis $H$ is said to be \emph{reduced} if, for all $\mathbf h_i + U \in H$, $\lc(h_i + U) = 1$ and $\mathbf h_i - \lm(\mathbf h_i + U) + U$ is a $\Bbbk$-linear combination of monomials in $\Mon_\preceq(R^d/U) \setminus \lm(\bar V)$.
    \end{enumerate}
\end{definition}

We can use Algorithm~\ref{alg:RelDivAlgo} below, adapted from \cite[Algo. 1]{HashemiOrthSeiler2021} for the module case, to compute remainders of module elements with respect to a subset of $R^d$ disjoint from a given submodule $U\subseteq R^d$.

\begin{algorithm}
    \caption{Relative division algorithm}
    \label{alg:RelDivAlgo}
    \begin{algorithmic}[1]
        \Statex \textbf{Input:} A monomial ordering $\prec$, a submodule $U \subseteq R^d$, a Gröbner basis $G$ of $U$, a finite set $H = \{ \mathbf{h}_1, \dotsc, \mathbf{h}_r \} \subseteq R^d \setminus \{ \mathbf{0} \}$ with $\NF(\mathbf{h}_i, U) = \mathbf{h}_i$ for all $i$, an element $\mathbf{f} \in R^d$.
        \Statex \textbf{Output:} A module element $\mathbf{p} \in R^d$ with support disjoint from $\langle\lm(U), \lm(H)\rangle$, polynomials $q_1, \dotsc, q_r \in R$ with $\mathbf{f} - \mathbf{p} - \sum_{i=1}^r q_i \mathbf{h}_i \in U$
        \Function{RelativeDivision}{$G, H, \mathbf{f}$}
            \State $\tilde{\mathbf f} \leftarrow \mathbf f$; $\mathbf p \leftarrow 0$
            \For{$i = 1, \dotsc, r$}
                \State $q_i \leftarrow 0$
            \EndFor
            \While{$\tilde{\mathbf f} \neq 0$}
                \If{$\lm(\tilde{\mathbf f}) \in \langle\lm(G) \rangle$}
                    \State Choose $\mathbf g \in G$ with $\lm(\mathbf g) | \lm(\tilde{\mathbf f})$
                    \State $\tilde{\mathbf f} \leftarrow \tilde{\mathbf f} - \frac{\lt(\tilde{\mathbf f})}{\lt(\mathbf g)} g$
                \ElsIf{$\lm(\tilde{\mathbf f}) \in \langle\lm(H)\rangle$}
                    \State Choose $\mathbf h_i \in H$ with $\lm(\mathbf h_i) | \lm(\tilde{\mathbf f})$
                    \State $q_i \leftarrow q_i + \frac{\lt(\tilde{\mathbf f})}{\lt(\mathbf h_i)}$;  $\tilde{\mathbf f} \leftarrow \tilde{\mathbf f} - \frac{\lt(\tilde{\mathbf f})}{\lt(\mathbf h_i)} \mathbf h_i$
                \Else
                    \State $\mathbf p \leftarrow \mathbf p + \lt(\tilde{\mathbf f})$;  $\tilde{\mathbf f} \leftarrow \tilde{\mathbf f} - \lt(\tilde{\mathbf f})$
                \EndIf
            \EndWhile
            \State \Return $(\mathbf p, q_1, \dotsc, q_r)$
        \EndFunction
    \end{algorithmic}
\end{algorithm}

\begin{remark}\label{rem:RelDivRem}
    The support of the quotient polynomial $q_k\in R$ associated to $\mathbf{h}_k$ computed in Algorithm \ref{alg:RelDivAlgo} is contained in the order ideal $\Mon(R)\setminus (\lm(U):\lm(\mathbf{h}_k))$.
    Since $H\cup G$ need not be a Gr\"{o}bner basis of $\langle H \rangle +U$, the module element $\mathbf{p}$ in the output of Algorithm \ref{alg:RelDivAlgo} is not uniquely determined by the input, but depends on the choices of $\mathbf{g}$ and $\mathbf{h}_i$ in the various reduction steps.
\end{remark}

\begin{definition}\label{def:RelReduces}
    If $\mathbf{p}$ is a possible output of Algorithm~\ref{alg:RelDivAlgo} for input $\mathbf{f},H,U,\prec$, then we write $\mathbf{f}\longrightarrow_{H,U,\prec}^{\ast} \mathbf{p}$ and say that $\mathbf{f}$ \emph{reduces to $\mathbf{p}$ with respect to $H$ modulo $U$}.
    We omit the reference to the ordering $\prec$ if no confusion can arise.
\end{definition}

\begin{definition}\label{def:RelGB}
    Let $U\subseteq V\subseteq R^d$ be submodules.
    \begin{enumerate}
        \item A finite subset $H \subseteq V \setminus U$ whose elements are supported on $\Mon(R^d)\setminus\lm(U)$ is said to be a \emph{Gröbner basis of $V$ relative to $U$} if $\langle \lm(H) \rangle + \lm(U) = \lm(V)$.
        \item A relative Gröbner basis $H$ is said to be \emph{minimal} if there are no divisibility relations among leading monomials of elements of $H$.
        \item A minimal relative Gröbner basis $H$ is said to be \emph{reduced} if, for all $\mathbf h \in H$, $\lc(\mathbf h) = 1$ and $\mathbf h - \lm(\mathbf h)$ is supported on $\Mon(R^d) \setminus \lm(V)$.
    \end{enumerate}
\end{definition}

Consider a reduced Gr\"{o}bner basis of $V\supseteq U$ relative to $U$.
It is a consequence of $\lm(U)\subseteq \lm(V)$ and of condition (ii) in Definition~\ref{def:RelGB} that $\NF(\mathbf{h},U)=\mathbf{h}$ for all $\mathbf{h}\in H$, i.e., in a reduced relative Gr\"{o}bner basis, all elements are given in normal form with respect to $U$.

\begin{remark}\label{rem:RelGB_QuotGB}
    It is not hard to see that $H$ is a Gr\"{o}bner basis of $V\supseteq U$ relative to $U$ if and only if $\{\mathbf{h}+U\mid \mathbf{h}\in H\}$ is a Gr\"{o}bner basis of $\bar{V}\coloneqq V/U$ in the sense of Definition~\ref{def:GB_ForNestedSubmodules}.
    Relative Gr\"{o}bner bases always exist, because any Gr\"{o}bner basis of $V$ is also a Gr\"{o}bner basis of $V$ relative to $U\subseteq V$.
\end{remark}

\begin{proposition}\label{prop:RelGBCompletionToClassicalGB}
    Let $H=\{\mathbf{h}_1,\dotsc ,\mathbf{h}_t\}\subseteq V$ be a finite set and $G$ a Gr\"{o}bner basis of $U\subseteq V$.
    Then the following statements are equivalent:
    \begin{enumerate}[label=(\roman*)]
        \item\label{RelGBCompletion(1)} $H$ is a Gr\"{o}bner basis of $V$ relative to $U$.
        \item\label{RelGBCompletion(2)} $H \cup G$ is a Gr\"{o}bner basis of $V$.
        \item\label{RelGBCompletion(3)} For every $\mathbf{f}\in V$, we have $\mathbf{f}\longrightarrow_{H,U,\prec}^{\ast} 0$.
        \item\label{RelGBCompletion(4)} Every $\mathbf{f}\in V$ has a {\em relative standard representation} of the form $\mathbf{f}=\mathbf{g}+\sum_{i=1}^t{q_i \mathbf{h}_i}$ where $\mathbf{g}\in U$ and $\lm(q_i \mathbf{h}_i)\preceq \lm(\mathbf{f})$ for each $i$ with $q_{i}\neq 0$.
    \end{enumerate}
\end{proposition}

\begin{proof}
    We show that \ref{RelGBCompletion(1)} $\Longrightarrow$ \ref{RelGBCompletion(2)} $\Longrightarrow$ \ref{RelGBCompletion(3)} $\Longrightarrow$ \ref{RelGBCompletion(4)} $\Longrightarrow$ \ref{RelGBCompletion(1)}.

    Assume \ref{RelGBCompletion(1)} holds.
    Then for each $\mathbf{m}\in\lm(V)\setminus \lm(U)$, there is an element $\mathbf{h}_i\in H$ such that $\lm(\mathbf{h}_i)$ divides $\mathbf{m}$.
    Moreover, for each $\mathbf{m}\in\lm(U)$, there is an element $\mathbf{g}\in G$ such that $\lm(\mathbf{g})$ divides $\mathbf{m}$.
    Thus, $H\cup G$ is a Gr\"{o}bner basis of $V$ and \ref{RelGBCompletion(2)} holds.

    Assume \ref{RelGBCompletion(2)} holds.
    Then, as each non-zero instance of $\tilde{\mathbf{f}}$ in Algorithm~\ref{alg:RelDivAlgo} applied to $\mathbf{f}$, $H$, and $G$ is an element of $V$, its leading monomial is in $\lm(V)=\langle \lm(H)\cup\lm(G) \rangle$.
    Thus $\mathbf{p}=\mathbf{0}$ will hold throughout the algorithm, showing that \ref{RelGBCompletion(3)} holds.

    Assume \ref{RelGBCompletion(3)} holds and let $\mathbf{f}\in V$.
    Then the output of Algorithm~\ref{alg:RelDivAlgo} applied to $\mathbf{f}$, $H$, and $G$ is $(\mathbf{0},q_1,\dotsc,q_t)$ such that $\mathbf{f}-\sum_{i=1}^r q_i \mathbf{h}_i\in U$.
    The relations $\lm(q_i \mathbf{h}_i)\preceq \lm(\mathbf{f})$ hold by the structure of the algorithm.
    Thus, \ref{RelGBCompletion(4)} holds.

    Assume \ref{RelGBCompletion(4)} holds.
    Let $\mathbf{f}\in V$ be a module element such that $\lm(\mathbf{f})\notin \lm(U)$.
    We may write $\mathbf{f}=\mathbf{g}+\sum_{i=1}^t{q_i \mathbf{h}_i}$ where $\mathbf{g}\in U$ and $\lm(q_i \mathbf{h}_i)\preceq \lm(\mathbf{f})$.
    As $\lm(\mathbf{f})\notin \lm(U)$, we obtain $\lm(\mathbf{g})\prec \lm(\mathbf{f})$.
    Thus there must exist at least one index $i$ such that $\lm(q_i\mathbf{h}_i)=\lm(\mathbf{f})$, implying $\lm(\mathbf{h}_i)|\lm(\mathbf{f})$.
    Thus, \ref{RelGBCompletion(1)} holds.
\end{proof}

We can thus use the Buchberger algorithm in $R^d$ to compute relative Gr\"{o}bner bases.
In this context, we mean the Buchberger algorithm that performs full reductions of all $S$-polynomials with respect to the current candidate set and does not delete any polynomials that have been added to that candidate set.

\begin{proposition}\label{prop:RelGB_BBAlgo}
    Let $U\subseteq V\subseteq R^d$ be two submodules and $F$ a finite generating set of $V$.
    Further, let $G$ be a Gr\"{o}bner basis of $U$ and call $H_V$ the Gr\"{o}bner basis of $V$ obtained from the set $(\NF(F,U)\setminus U) \cup G$ via Buchberger's algorithm.
    Then $H\coloneqq H_V\setminus U$ is a Gr\"{o}bner basis of $V$ relative to $U$.
\end{proposition}
\begin{proof}
    First note that, since $G$ generates $U$, the set $(\NF(F,U)\setminus U)\cup G$ still generates $V$.
    As $G$ is a subset of the input set for Buchberger's algorithm, all elements of $U$ reduce to zero with respect to $G$.
    Thus the algorithm will only add elements of $V\setminus U$ to the candidate set, and all added elements will be in normal form with respect to $U$ and the chosen monomial ordering.
    Thus, the output $H_V$ can be written as $H\cup G$, where $H\coloneqq H_V\setminus U$, and for each $\mathbf{m}\in\lm(V)\setminus\lm(U)$ there is an element $\mathbf{h}\in H$ such that $\lm(\mathbf{h})$ divides $\mathbf{m}$.
    Thus, $H$ is a Gr\"{o}bner basis of $V$ relative to $U$.
\end{proof}

\TheoremA*

\begin{proof}
    By Proposition~\ref{prop:RelGB_BBAlgo}, there exists a Gr\"{o}bner basis of $V$ relative to $U$.
    By a complete interreduction and normalization of the elements of that basis (in intermediate steps reducing elements to their normal form with respect to $U$), we obtain a reduced relative Gr\"{o}bner basis.
    Now let $H$ and $H'$ be two reduced Gr\"{o}bner bases of $V$ relative to $U$.
    Then for each minimal generator $\mathbf{m}$ of $\lm(V)$ that is not contained in $\lm(U)$ there is exactly one element $\mathbf{h}_{\mathbf{m}}\in H$ such that $\lm(\mathbf{h}_{\mathbf{m}})=\mathbf{m}$ and exactly one element $\mathbf{h}'_{\mathbf{m}}\in H'$ such that $\lm(\mathbf{h}'_{\mathbf{m}})=\mathbf{m}$.
    This induces a bijection $H\to H',\, \mathbf{h}_{\mathbf{m}}\mapsto \mathbf{h}'_{\mathbf{m}}$.
    Moreover, as leading coefficients are $1$ in a reduced relative Gr\"{o}bner basis, the element $\mathbf{d}_{\mathbf{m}}\coloneqq\mathbf{h}_{\mathbf{m}}-\mathbf{h}'_{\mathbf{m}}\in V$ is supported on module monomials outside $\lm(V)$, implying $\mathbf{d}_{\mathbf{m}}=\mathbf{0}$.
    Hence, $H=H'$, i.e., the reduced Gröbner basis obtained by complete interreduction and normalization is unique.
\end{proof}

\begin{remark}\label{rem:Furthermore_in_Thm_A}
    If $G_V$ is a Gr\"{o}bner basis of $V$, then $H\coloneqq \NF(G_V,U)\setminus \{\mathbf{0}\}$, the normal form of $G_V$ relative to $U$, is a Gr\"{o}bner basis of $V$ relative to $U$. Indeed, for each minimal generator $\mathbf{m}$ of $\lm(V)$ that is not contained in $\lm(U)$ there is an element $\mathbf{g}_{\mathbf{m}}\in G_V$ with $\lm(\mathbf{g}_\mathbf{m})=\mathbf{m}$.
    As $\mathbf{m}\notin \lm(U)$, taking the normal form does not change the leading monomial. In particular, $\mathbf{g}'_{\mathbf{m}}\neq \mathbf{0}$.
    Thus, $H$ contains $\mathbf{g}'_{\mathbf{m}}$, and thus it is a relative Gr\"{o}bner basis. If $G_V$ is reduced, then $H$ is the reduced Gr\"{o}bner basis of $V$ relative to $U$.
\end{remark}

\begin{remark}\label{rem:complexity_rel_GB}
    The complexity of computing relative Gr\"{o}ner bases is governed by the complexity of computing classical Gr\"{o}bner bases. Computing a Gr\"{o}bner basis of $V$ relative to $U$ amounts to a stepwise computation of a Gr\"{o}bner basis of $V$, where a Gr\"{o}bner basis of $U$ is obtained in an intermediate step.
\end{remark}

Assuming that the modules $U$ and $V$ are homogeneous, for each of the modules $U$, $V$, and $V/U$ respectively, all minimal generating sets have the same cardinality.
However, it is important to distinguish between the concepts of minimal generating sets and reduced relative Gr\"{o}bner bases, even if the module $U$ is monomial, as illustrated in the following example.

\begin{example}\label{ex:ReducedRelGB_NotMinimalGens}
    Let $R=\Bbbk[X,Y]$, and use the degree reverse lexicographic monomial ordering with $X\prec Y$.
    Consider the submodules $U\subseteq V\subseteq R^1=R$, where
    \[ U\coloneqq\langle X,Y \rangle^5 \text, \]
    and
    \[ V\coloneqq U+\langle Y^3,XY^2 +X^3 \rangle \text. \]
    Then $A\coloneqq\{Y^3+U,XY^2 +X^3+U\}$ is a minimal generating set of $V/U$, but an application of Buchberger's algorithm to $A\cup \{X^5,X^4 Y,\dotsc,Y^5\}$ shows that the reduced Gr\"{o}bner basis of $V/U$ is
    \[ H\coloneqq A\cup\{ X^3Y+U \} \text. \]
\end{example}

When all modules considered are monomial, we have the following
result.

\begin{proposition}\label{prop:MinGen<->RedRelGB_ForMonomial_UandV}
    Let $U\subseteq V\subseteq R^d$ be monomial submodules of $R^d$.
    Then the reduced Gr\"{o}bner basis of $V/U$ is a minimal generating set of $V/U$.
\end{proposition}

\begin{proof}
    Any set of monomials in $R^d$ is a Gr\"{o}bner basis of the submodule it generates.
    Thus the minimal monomial generating set of $V/U$ is also a Gr\"{o}bner basis, which is clearly reduced.
\end{proof}

\subsection{Relative Schreyer's presentations}%
\label{sect:rel-free}

 In this section, we establish a generalized Schreyer's theorem for constructing a free presentation for a submodule $M=V/U$ of $R^d/U$, and by iterating, a method to obtain a free resolution for $M$. While we formulate our constructions in the ungraded setting, the results of this subsection are also valid for $n$-graded $R$-modules if all statements that involve powers of $R$ are replaced by direct sums of grade-shifted copies of $R$. See also Remark~\ref{rem:triv-grad}.

Let $U\subseteq V\subseteq R^d$ be nested submodules of $R^d$.
Then $V/U$ is an $R$-module via the structure map
\[
R\times (V/U)\ni(r,\mathbf{f}+U)\mapsto (r\mathbf{f})+U \text.
\]

\begin{example}\label{ex:FreeCover}
    Consider $R=\Bbbk[X]$, the polynomial ring in one variable.
    In the rank $1$ free module $R^1$, consider the standard graded monomial submodules $U\coloneqq\langle X^2\rangle \subseteq \langle X\rangle\eqqcolon V$.
    Then we obtain the free cover $\Sigma ^1 R\to V/U$, $\mathbf{e}_1\mapsto X+ U$ (we interpret this as a homomorphism of $R$-modules).
    The kernel is $\langle X\rangle$, a free $R$-module.
    This gives a free resolution $0\to \Sigma^2 R\to \Sigma^1 R\to V/U\to 0$.
\end{example}
From a ($n$-graded) Gr\"{o}bner basis $G=\{\mathbf{g}_1,\dotsc,\mathbf{g}_s\}$ of $V$ relative to $U$, we can obtain a Gr\"{o}bner basis of the kernel of the ($n$-graded) presentation map
\[ R^s\to V/U, \quad \mathbf{e}_i\mapsto \mathbf{g}_i+U \text, \]
which is a homomorphism of $R$-modules.
To compute the kernel, we use a generalization of Schreyer's construction, inspired by the works~\cite{LaScalaStillman1998,Schreyer1980MastersThesis,EroecalEtAl2016Syzygies}. We need the following notation.

\begin{definition}\label{def:Projection}
    For an index set $I\subseteq \{1,\dotsc,d\}$, write $\pi_I$ for the $R$-linear projection mapping
    \[ \pi_I\colon R^d \longrightarrow R^I\text, \quad \sum_{i=1}^d f_i \mathbf{e}_i\longmapsto \sum_{i\in I} f_i \mathbf{e}_i\; \text. \]
\end{definition}

\begin{definition}\label{def:ProjectionOfSyzygies}
    Let $U\subseteq V\subseteq R^d$ be submodules of $R^d$, let $G_V=\{\mathbf{g}_1,\dotsc,\mathbf{g}_t\}$ be a Gr\"{o}bner basis of $V$ relative to $U$ and let $G_U=\{\mathbf{g}_{t+1},\dotsc,\mathbf{g}_s\}$ be a Gr\"{o}bner basis of $U$.
    Let $G\coloneqq G_V\cup G_U$ and let $G^{(1)}=\{\mathbf{S}_{ij}\mid 1\leq i<j\leq s\}$ be the Gr\"{o}bner basis of $\Syz(G)$ with respect to $\prec_G$ obtained by Schreyer's construction (Theorem~\ref{thm:SchreyerConstruction}), using Algorithm~\ref{alg:RelDivAlgo} for standard representations.
    Then we write
    \begin{equation}\label{eq:ProjectedSyzygies}
        \overline{\mathbf{S}_{ij}}\coloneqq\pi_{\{1,\dotsc,t\}}(\mathbf{S}_{ij})
    \end{equation}
    and $H^{(1)}\coloneqq\{ \overline{\mathbf{S}_{ij}}\mid 1\leq i <j\leq s\;\wedge\; i\leq t \}$.
\end{definition}

The set $H^{(1)}$ is the relative analog to the Gr\"{o}bner basis $G^{(1)}$ of the first syzygy module for a Gr\"{o}bner basis $G$.

\TheoremB*

\begin{proof}
    Let $\{\mathbf{h}_1,\dotsc,\mathbf{h}_t\}$ be a Gr\"{o}bner basis (respectively, the unique reduced) of $V$ relative to $U$ and write $H=\{\mathbf{h}_i+U\mid 1\leq i\leq t\}$. It suffices to show that $H^{(1)}\subseteq R^t$ is a Gr\"{o}bner basis of the $R$-module $\Syz(H)$ because then Gr\"{o}bner bases $H^{(i)}$, $i\geq 2$ are (uniquely) obtained using standard techniques for module Gr\"{o}bner bases.
    Consider the presentation map $\psi\colon R^t\to V/U$, $\mathbf{e}_i\mapsto \mathbf{h}_i+U$, which is a homomorphism of $R$-modules.
    Then, we claim that a Gr\"{o}bner basis of $\ker(\psi)$ is given by $H^{(1)}$ as in Equation~\eqref{eq:ProjectedSyzygies}, where this set is obtained by appending a Gr\"{o}bner basis $G_U=\{\mathbf{g}_{t+1},\dotsc,\mathbf{g}_s\}$ of $U$ to $G_V\coloneqq\{\mathbf{h}_1,\dotsc,\mathbf{h}_t\}$ to obtain $G\coloneqq G_V\cup G_U$ and then applying the construction of Definition~\ref{def:ProjectionOfSyzygies} to $G$.

    For a syzygy $\mathbf{S}=\sum_{i=1}^t f_i\cdot \mathbf{e}_i\in \Syz(H)$, note that $\mathbf{f}\coloneqq\sum_{i=1}^t f_i \mathbf{h}_i \in U$.
    There is a standard representation $\mathbf{f}=\sum_{j=1}^{s-t} f_j \mathbf{g}_{t+j}$, and then $\mathbf{T}\coloneqq\sum_{i=1}^s f_i\mathbf{e}_i+ \sum_{j=1}^{s-t} g_j \mathbf{e}_{t+j}\in \Syz(G)$ is a syzygy with $\lm_{\prec_G}(\mathbf{T})=\lm_{\prec_H}(\mathbf{S})$.
    Thus $\lm_{\prec_H}(\Syz(H))\subseteq \lm_{\prec_G}(\Syz(G))$; more precisely, for each $\mathbf{S}\in\Syz(H)$ there is a syzygy $\mathbf{S}_{ij}\in G^{(1)}$ with $i\leq t$ such that $\lm(\mathbf{S}_{ij})|\lm(\mathbf{S})$.

    Conversely, let $\mathbf{S}_{ij}\in G^{(1)}$ with $i\leq t$ be any syzygy.
    It is of the form $\mathbf{S}_{ij}=\sum_{i=1}^s f_i\mathbf{e}_i+ \sum_{j=1}^{s-t} g_j \mathbf{e}_{t+j}$ and its leading monomial is supported on the $i$-th module component.
    Since $\sum_{j=1}^{s-t} g_j \mathbf{g}_{t+j}\in U$, we have that $\pi_{\{1,\dotsc,t\}}(\mathbf{S}_{ij})\in \Syz(H)$.
    Thus $\lm(\pi_{\{1,\dotsc,t\}}(G^{(1)}))$ generates $\lm(\Syz(H))$.
    As $\pi_{\{1,\dotsc,t\}}(G^{(1)})=H^{(1)}$, we are done.
\end{proof}

\begin{example}\label{Ex:Non_minimality_Syzygy_GB}
    The relative Gr\"{o}bner basis $H^{(1)}$ is in general not minimal, let alone reduced. For example, consider the $3$-graded submodules $U=\{0\}\subseteq \langle XY,YZ,XZ \rangle=V\subseteq R^{1}$, where $R=\Bbbk[X,Y,Z]$.
    Set $H=\{XY,YZ,XZ\}$.
    Then the three vectors $\mathbf{S}_{12}=Z\mathbf{e}_1-X\mathbf{e_2}$, $\mathbf{S}_{13}=Z\mathbf{e}_1-Y\mathbf{e_3}$, and $\mathbf{S}_{23}=X\mathbf{e}_2-Y\mathbf{e_3}$ form a non-minimal Gr\"{o}bner basis of $\Syz(H)$.
    Deleting $\mathbf{S}_{13}$, we obtain a minimal Gr\"{o}bner basis $\{ \mathbf{S}_{12},\mathbf{S}_{23} \}$, which is not reduced, because the tail term $-X\mathbf{e}_2$ of $\mathbf{S}_{12}$ is divisible by $\lm(\mathbf{S}_{23})$.
    An interreduction yields the reduced Gr\"{o}bner basis $\{ \mathbf{S}_{12}+\mathbf{S}_{23},\mathbf{S}_{23}\}=\{Z\mathbf{e}_1-Y\mathbf{e_3},X\mathbf{e}_2-Y\mathbf{e_3}\}$.
    Note that we could also have arrived at the reduced Gr\"{o}bner basis by deleting $\mathbf{S}_{12}$ instead of $\mathbf{S}_{13}$ during minimization.
\end{example}

We thus have a procedure for obtaining a Gr\"{o}bner basis of the kernel of a presentation $\psi$ of $V/U$.
As $\ker \psi$ is an $R$-submodule and we have a Gr\"{o}bner basis $H^{(1)}$ of it, we can use Schreyer's construction to compute a Gr\"{o}bner basis $H^{(2)}$ for the kernel of a presentation of $\ker(\psi)$.
By iterating this construction, we obtain a free resolution of $V/U$ over $R$.

\begin{remark}\label{rem:Minimization_of_syzygy_GB}
    Under the hypotheses of Theorem~\ref{thm:KernelOfPresentation:intro}, one can extract a minimal relative Gr\"{o}bner basis of $\Syz(H)$ from $H^{(1)}$ defined by taking a subset $H_{\min}^{(1)}\subseteq H^{(1)}$ such that the set of leading monomials $\{\lm(\mathbf{v})\mid \mathbf{v}\in H_{\min}^{(1)}\}$ minimally generates $\lm(\Syz(H))$.
\end{remark}

\section{Multigraded presentations and resolutions}%
\label{sect:mgraded-pres}

In the previous sections, we described a method for computing a presentation of the $R$-submodule $V/U$ of $R^d/U$.
This method extends naturally to the multigraded setting by replacing the free module $R^d$ with a multigraded free module $F$. In this section, we focus on the multigraded module setting.

In the first subsection, we introduce the monomialization operation that associates to each $n$-graded $R$-module supported on $\mathbb{N}^n$, an $R$-module in such a way  that every shifted copy $\Sigma^a R$ of $R$ corresponds to the principal ideal $\langle X^a\rangle $ of $R$.

In the second subsection, we develop a general procedure for computing a presentation of the homology module, $\ker g/\im f$, of any chain complex $L\xrightarrow{f}M\xrightarrow{g}N$ of torsion-free $n$-graded modules. As a special case, this construction  produces a presentation of $V/U\subseteq F/U$, viewed as the homology of the complex of torsion free modules, $U \hookrightarrow V \hookrightarrow F$.
Computing a presentation of homology is the first step toward constructing a free resolution of homology: by iteratively applying the same method to the syzygy modules of the presentation, one ultimately obtains a multigraded free resolution of the homology module.

In the third subsection, we focus on the special case where $F/U$ is a cofree module\footnote{Recall
that a cofree module is a finitely cogenerated injective module.} and we obtain an algorithm for computing a presentation--and thus a resolution by iterating the procedure--of any submodule $V/U$ of a cofree module $F/U$.

\subsection{Monomialization of multigraded modules}%
\label{ssect:monomialization}

Let $\alpha_1, \dotsc, \alpha_j \in \mathbb Z^n$.
We consider submodules $N$ of the free $R$-module $\bigoplus_{i=1}^d \Sigma^{\alpha_j} R$.
In order to apply Gröbner bases and relative Gröbner bases to the module $N$ and its submodules, it is necessary to analytically express $N$ as a submodule of $R = \bigoplus_{j=1}^d R$ without changing any module-theoretic properties of $N$.
For this purpose, we introduce the so-called \emph{monomialization of $N$}, denoted by $N^{\mon}$, as follows.

For an exponent vector $\alpha \in \mathbb N^n$ we consider the monomialization map
\[ \mu_\alpha\colon \Sigma^{\alpha} R \longrightarrow R \text, \quad f \longmapsto X^{\alpha} \cdot f \text. \]
The map $\mu_a$ is a graded monomorphism of degree $0$, which induces an isomorphism $\Sigma^\alpha R \cong (X^\alpha)$.
Further, for a sequence $A = \alpha_1, \dotsc, \alpha_d \in \mathbb N^n$, let $\mu_A\colon \bigoplus_{j=1}^d \Sigma^{\alpha_j} R \to R^d$ be the map defined by
\[ \mu_A(f_1, \dotsc, f_d) \coloneqq (\mu_{\alpha_1}(f_1), \dotsc, \mu_{\alpha_t}(f_d)) \text, \qquad \text{for $(f_1, \dotsc, f_d) \in \bigoplus_{j=1}^d \Sigma^{\alpha_j} R$.} \]

\begin{definition}\label{def:Monomialization}
    For an element $\mathbf f = (f_1, \dotsc, f_d) \in \bigoplus_{j=1}^d \Sigma^{\alpha_j} R$, the image $\mathbf f^{\mon} \coloneqq \mu_A(\mathbf f)$ is called the \emph{monomialization of~$\mathbf f$}.

    Further, for a graded submodule $N \subseteq \bigoplus_{i=1}^d \Sigma^{\alpha_j} R$ the image of $N$ under the monomialization map $\mu_A$ is denoted by $N^{\mon}$ and called the \emph{monomialization of~$N$}.
\end{definition}

\begin{lemma}\label{lem:GeneratingSets_Monomialization}
    Suppose the submodule $N \subseteq \bigoplus_{i=1}^d \Sigma^{\alpha_i} R$ is generated by $\mathbf f_1, \dotsc, \mathbf f_t$.
    Then $N^{\mon}$ is generated by $\mathbf f_1^{\mon}, \dotsc, \mathbf f_t^{\mon}$.
    \begin{proof}
        The monomialization map $\mu_A$ restricts to an isomorphism $N \cong N^{\mon}$ by definition.
    \end{proof}
\end{lemma}

\begin{example}
    We reconsider the module~$M$ from Example~\ref{ex:flange-1}.
    Notice that $M$ can be realized as a quotient module $V/U$ for a chain of submodules $U \subseteq V \subseteq F$ where
    \[ F \coloneqq \Sigma^{(1, 0)} R \oplus \Sigma^{(0, 1)} R \oplus \Sigma^{(2, 0)} R \oplus \Sigma^{(1, 1)} R^2 \oplus \Sigma^{(2, 1)} R \]
    is a graded free module with standard basis $\mathbf e_1, \dotsc, \mathbf e_6$.
    Then we can choose generators of $V$ as
    \begin{align*}
        \mathbf {g}_1 &\coloneqq \mathbf {e}_1 + \mathbf {e}_3 + \mathbf {e}_4 + \mathbf {e}_6 \text, &\quad \mathbf {g}_2 &\coloneqq \mathbf {e}_2 + \mathbf {e}_5 + \mathbf {e}_6 \text, \\
        \mathbf {g}_3 &\coloneqq \mathbf {e}_3 + \mathbf {e}_6 \text, &\quad \mathbf {g}_4 &\coloneqq \mathbf {e}_4 + \mathbf {e}_6 \text, \\
        \mathbf {g}_5 &\coloneqq \mathbf {e}_5 + \mathbf {e}_6 \text, &\quad \mathbf {g}_6 &\coloneqq \mathbf {e}_6\text.
    \end{align*}
    Application of the monomialization to the submodule $V$ and to the generators $\mathbf g_1, \dotsc, \mathbf g_6$ gives the submodule $V^{\mon} \subseteq R^6$ with generators
    \begin{align*}
        \mathbf g_1^{\mon} &= X_1 (\mathbf e_1 + \mathbf e_3 + \mathbf e_4 + \mathbf e_6) \text, &\quad
        \mathbf g_2^{\mon} &= X_2 (\mathbf e_2 + \mathbf e_5 + \mathbf e_6) \text, \\
        \mathbf g_3^{\mon} &= X_1^2 (\mathbf e_3 + \mathbf e_6) \text, &\quad
        \mathbf g_4^{\mon} &= X_1 X_2 (\mathbf e_4 + \mathbf e_6) \text, \\
        \mathbf g_5^{\mon} &= X_1 X_2 (\mathbf e_5 + \mathbf e_6) \text, &\quad
        \mathbf g_6^{\mon} &= X_1^2X_2 \mathbf e_6 \text.
    \end{align*}
\end{example}

\subsection{Presentation of homology of torsion-free chain complexes}

A fundamental task in multiparameter topological data analysis (TDA) is the efficient computation of a free presentation of homology,
\[ V/U=\ker g/\im f \text, \text{ where }U\coloneqq\im f\text{ and }V\coloneqq\ker g,\]
of a given chain complex
\[ \begin{tikzcd}[ampersand replacement=\&]
    L \rar["f"] \& M \rar["g"] \& N
\end{tikzcd} \]
of $n$-graded torsion-free modules.

In this section, we present a direct solution to that problem by embedding $V/U$ as a submodule of $F/U$ for a proper graded free module $F$, and utilizing the relative Schreyer's theorem, see Theorem \ref{thm:KernelOfPresentation:intro}, to compute the presentation.

Since the modules $L$, $M$ and $N$ are assumed to be finitely generated, they can be covered by finite rank graded free modules.
Further, since these modules are assumed to be torsion-free, they admit embeddings into finite rank graded free modules.
That is, there exist epimorphisms~$\pi_1, \pi_2, \pi_3$ and monomorphisms~$\iota_1, \iota_2, \iota_3$,
\[ \begin{tikzcd}[ampersand replacement=\&]
    \displaystyle\bigoplus_{j=1}^{t_1} \Sigma^{\alpha_j} R \& \displaystyle\bigoplus_{j=1}^{t_2} \Sigma^{\beta_j} R \& \displaystyle\bigoplus_{j=1}^{t_3} \Sigma^{\gamma_j} R \\
    L \rar["f"] \& M \rar["g"] \& N \\
    \displaystyle\bigoplus_{i=1}^{s_1} \Sigma^{\alpha'_i} R \& \displaystyle\bigoplus_{i=1}^{s_2} \Sigma^{\beta'_i} R \& \displaystyle\bigoplus_{i=1}^{s_3} \Sigma^{\gamma'_i} R \rlap{\text,}
    \arrow[from=1-1, to=2-1, two heads, "\pi_2"]
    \arrow[from=1-2, to=2-2, two heads, "\pi_1"]
    \arrow[from=1-3, to=2-3, two heads, "\pi_0"]
    \arrow[from=2-1, to=3-1, hook, "\iota_2"]
    \arrow[from=2-2, to=3-2, hook, "\iota_1"]
    \arrow[from=2-3, to=3-3, hook, "\iota_0"]
\end{tikzcd} \]
for some $\alpha_j, \alpha'_i$, $\beta_j, \beta'_i$, and $\gamma_j, \gamma'_{i} \in \mathbb Z^n$.

\TFHomology*

\begin{proof}[Proof of Theorem~\ref{thm:homology presentation}]
The main technicality here is to extract generators of $V=\ker g$, as generators of $U$ can be easily obtained, since $U=\im f$ is also equal to the image of the free-to-free homomorphism
\[ \begin{tikzcd}[ampersand replacement=\&]
    \displaystyle\bigoplus_{j=1}^{t_1} \Sigma^{\alpha_j} R \rar["\pi_2", two heads] \& L \rar["f"] \& M \rar["\iota_1", hook] \& \displaystyle\bigoplus_{i=1}^{s_2} \Sigma^{\beta'_i} R \rlap{\text,}
\end{tikzcd} \]
and thus it can be encoded by a graded matrix.
If we find a generating set for $\ker g$, then we can compute a Gr\"obner basis of $V$ relative to $U$ and then directly apply the relative Schreyer's theorem, Theorem \ref{thm:KernelOfPresentation:intro}, to obtain a presentation of homology.

We  compute generators of $\ker g$ by using implicitly the classical Schreyer's theorem. Classical Schreyer’s theorem gives us a way to compute kernels of homomorphisms of free modules (see Remark \ref{rmk:Schreyer thm for obtaining kernel of morphism of free modules}).
Free modules are torsion-free, but not vice versa. For the general case of torsion-free modules, we express the kernel as
\[ \ker g = \pi_1 \Bigg( \ker\bigg( \begin{tikzcd}[ampersand replacement=\&]
    \displaystyle\bigoplus_{j=1}^{t_2} \Sigma^{\beta_j} R \rar[two heads, "\pi_1"] \& M \rar["g"] \& N \rar[hook, "\iota_0"] \& \displaystyle\bigoplus_{i=1}^{s_3} \Sigma^{\gamma'_i} R
\end{tikzcd} \bigg) \Bigg) \text. \]
That is, first we consider the free-to-free homomorphism
\[ \begin{tikzcd}[ampersand replacement=\&]
    \displaystyle\bigoplus_{j=1}^{t_2} \Sigma^{\beta_j} R \rar["\pi_1", two heads] \& M \rar["g"] \& N \rar["\iota_0", hook] \& \displaystyle\bigoplus_{i=1}^{s_3} \Sigma^{\gamma'_i} R \rlap{\text,}
\end{tikzcd} \]
and compute a Gr\"obner basis for the kernel $\ker(\iota_0 \circ g \circ \pi_1)$
using the classical Schreyer theorem.
Then applying $\pi_1$ to these generators yields a generating set for $\ker g$.
\end{proof}

\begin{remark}%
    \label{rem:homology resolution}
    Using Theorem~\ref{thm:homology presentation}, one can initiate the computation of a free resolution of the homology module of a complex of torsion free $n$-graded modules. The resolution can then be minimized using the pruning method which will be discussed in Section~\ref{sect:minimal-res}.
\end{remark}

We provide an explicit example that demonstrates how to compute a free presentation of the homology module using our technique.

\begin{example}
    We consider the bifiltration~$X$ from Figure~\ref{fig:flange-bifilt}.
    Our aim is to compute a free presentation of the persistent simplicial homology module $H_1(X)$ restricted to the $3 \times 2$-sublattice of $\mathbb Z^n$ by using Theorem~\ref{thm:homology presentation}.
    The homomorphisms $\pi_2$, $\pi_1$, $\iota_1$, and $\iota_0$ in the diagram
    \[ \begin{tikzcd}[ampersand replacement=\&, column sep=small]
        \Sigma^{\alpha'_1} R \& \bigoplus_{i=1}^6 \Sigma^{\beta'_i} R \& \\
        C_2(X) \& C_1(X) \& C_0(X) \\
        \& \bigoplus_{i=1}^5 \Sigma^{\beta_i} R \& \bigoplus_{i=1}^4 \Sigma^{\gamma_i} R
        \arrow[from=1-1, to=2-1, "\pi_2"]
        \arrow[from=1-2, to=2-2, "\pi_1"]
        \arrow[from=2-1, to=2-2, "\partial_2"]
        \arrow[from=2-2, to=2-3, "\partial_1"]
        \arrow[from=2-2, to=3-2, "\iota_1"]
        \arrow[from=2-3, to=3-3, "\iota_0"]
    \end{tikzcd} \quad\text{with}\quad \left\{\begin{aligned}
        (0, 0) &\eqqcolon \beta_1 \eqqcolon \beta_2 \eqqcolon \beta_3 \\
        &\eqqcolon \beta_1' \eqqcolon \beta'_2 \eqqcolon \gamma_1 \\
        &\eqqcolon \gamma_2 \eqqcolon \gamma_3 \eqqcolon \gamma_4 \text, \\
        (1, 0) &\eqqcolon \beta_4 \eqqcolon \beta'_3 \eqqcolon \beta'_5 \text, \\
        (0, 1) &\eqqcolon \beta_5 \eqqcolon \beta'_4 \eqqcolon \beta'_6 \text, \\
        (2, 1) &\eqqcolon \alpha'_1 \text,
    \end{aligned} \right. \]
    yield morphisms of free modules, represented by the graded matrices below
    \begin{gather*}
         \iota_0 \circ \partial_1 \circ \pi_1  = \begin{bmatrix}
            -1 & -1 & 0 & 0 & 0 & 0 \\
            1 & 0 & -X_1 & -X_2 & -X_1 & 0 \\
            0 & 0 & X_1 & X_2 & 0 & -X_2 \\
            0 & 1 & 0 & 0 & X_1 & X_2
        \end{bmatrix} \text, \\
         \iota_1 \circ \pi_1  = \begin{bmatrix}
            1 & 0 & 0 & 0 & 0 & 0 \\
            0 & 1 & 0 & 0 & 0 & 0 \\
            0 & 0 & X_1 & X_2 & 0 & 0 \\
            0 & 0 & 0 & 0 & 1 & 0 \\
            0 & 0 & 0 & 0 & 0 & 1
        \end{bmatrix} \text{, } \iota_1 \circ \partial_2 \circ \pi_2 = \begin{bmatrix}
            0 \\
            0 \\
            X_1^2 X_2 \\
            -X_1 X_2 \\
            X_1^2
        \end{bmatrix}
    \end{gather*}
    where the free modules are endowed with the standard (graded) bases.

    We compute the kernel of $\iota_0 \circ \partial_1 \circ \pi_1$
    by computing a Gr\"{o}bner basis $G$ of the set of columns of $\partial_1$, computing the syzygies of $G$, extracting the syzygies $S$ of $G$, and finally computing a Gr\"{o}bner basis $H$ of $S$. All Gr\"{o}bner basis computations will be done with respect to a POT module monomial ordering with $\mathbf e_1 \succ \mathbf e_2 \succ \dotsb$.
    Following this pipeline, we obtain
    \[
    G=\begin{bmatrix}
        1 & 0 & 0 & 0 \\
        -1 & 1 & 0 & 0 \\
        0 & 0 & X_1 & X_2 \\
        0 & -1 & -X_1 & -X_2 \\
    \end{bmatrix},
    \]
    after application of Buchberger's algorithm and full reduction.
    In particular, we have $G= \partial_1 \cdot A$ and $\partial_1 = G \cdot B$ where
    \[
    A=\begin{bmatrix}
        -1 & 1 & X_1 & 0 \\
        0 & -1 & -X_1 & 0 \\
        0 & 0 & 1 & 0 \\
        0 & 0 & 0 & 0 \\
        0 & 0 & 0 & 0 \\
        0 & 0 & 0 & -1 \\
    \end{bmatrix},\quad
    B=\begin{bmatrix}
        -1 & -1 & 0 & 0 & 0 & 0 \\
        0 & -1 & -X_1 & -X_2 & -X_1 & 0 \\
        0 & 0 & 1 & 0 & 0 & 0 \\
        0 & 0 & 0 & 1 & 0 & -1 \\
    \end{bmatrix}.
    \]
    Using Schreyer's construction, the syzygy module of $G$ is generated by
    \[ \mathbf w \coloneqq X_2 \mathbf e_3 - X_1 \mathbf e_4 \text. \]
    Thus the syzygies of $\partial_1$ are given by the columns of the following matrix $S$, which contains $A\cdot\mathbf{w}$ and the non-zero columns of $\mathbb{I}_6-A\cdot B$,
    \[
    S=\begin{bmatrix}
        X_1X_2 & X_2 & X_1 \\
        -X_1X_2 & -X_2 & -X_1 \\
        X_2 & 0 & 0 \\
        0 & 1 & 0 \\
        0 & 0 & 1 \\
        -X_1 & 1 & 0 \\
    \end{bmatrix}.
    \]
    Applying Buchberger's algorithm, we find the Gr\"{o}bner basis consisting of the columns of
    \[
    H=\begin{bmatrix}
        X_1 & X_2 & 0 & 0 \\
        -X_1 & -X_2 & 0 & 0 \\
        0 & 0 & X_2 & 0 \\
        0 & 1 & -X_1 & X_1 \\
        1 & 0 & 0 & -X_2 \\
        0 & 1 & 0 & X_1 \\
    \end{bmatrix}.
    \]
    The only non-trivial syzygy of this Gr\"{o}bner basis is
    \[ X_2 \mathbf e_1 - X_1 \mathbf e_2 + \mathbf e_4 \text, \]
    which implies that the first three columns of $H$ are a minimal generating set of the kernel of $\iota_0 \circ \partial_1 \circ \pi_1$.

    We apply the map $\iota_1\circ \pi_1$ to the columns of $H$ and obtain (omitting zero columns)
    \[
    H'=\begin{bmatrix}
        X_1 & X_2 & 0 \\
        -X_1 & -X_2 & 0 \\
        0 & X_2 & X_1X_2 \\
        1 & 0 & -X_2 \\
        0 & 1 & X_1 \\
    \end{bmatrix}.
    \]
    The columns of $H'$ form a reduced Gr\"{o}bner basis of $V\coloneqq\iota_1 (\ker \partial_1)$ relative to $U\coloneqq\iota_1 (\im \partial_2)$ for a POT module monomial ordering. Hence, the syzygy module of $H'$ is generated by the columns of
    \[
    S'=\begin{bmatrix}
        X_2 & 0 \\
        -X_1 & 0 \\
        1 & X_1
    \end{bmatrix},
    \]
    where the first column is an $S$-polynomial and the second results from annihilation of the leading module term of the third column of $H'$ modulo $\lt U$. Finally one can extract the minimal generating set of $V/U$ consisting of the first two columns of $H'$, with $\begin{bmatrix}
        X_1X_2 & -X_1^2
    \end{bmatrix}^\top$ as cyclic generator of its syzygy module.
\end{example}

\subsection{Presentations of submodules of cofree modules}%
\label{sect:gb-free-cofree}

Throughout the section we assume that every $R$-module is graded, finitely generated, and finitely supported.
For such an $R$-module~$M$ we consider its flat-injective presentations $\varphi\colon F \to E$, see Definition~\ref{def:flange}, where $F$ is free of finite rank and $E$ is \emph{cofree of finite corank}, which means that $E \cong \hom_R(Q, E(\Bbbk))$ for some free $R$-module $Q$ of finite rank.
Then, because such flat-injective presentations admit free-injective matrices and these represent the generators of~$M$ as elements of an injective module, we can apply the techniques of relative Gröbner bases from Section~\ref{sect:rel-gb} to flat-injective presentations.
In particular, we provide definitions of Gröbner basis algorithms (Buchberger's algorithm and Schreyer's algorithm) that are specialized to flat-injective presentations, see Algorithm~\ref{alg:BuchbergerFlange} and Algorithm~\ref{alg:FreePresentation}.

Our existence assumption on specific flat-injective presentations is a direct consequence of our finiteness assumptions on $R$-modules.

\begin{lemma}%
    \label{lem:fg-fs-flange}
    Let $M$ be a finitely generated and finitely supported module.
    \begin{enumerate}
        \item\label{lem:fg-fs-flange:1} There exist $\alpha_1, \dotsc, \alpha_s \in \mathbb Z^n$ such that $M$  embeds into $\bigoplus_{i = 1}^s \Sigma^{\alpha_i} E(\Bbbk)$.
        \item\label{lem:fg-fs-flange:2} There is an epimorphism  $\bigoplus_{j=1}^t \Sigma^{\beta_j} R\twoheadrightarrow M$, for some $\beta_1, \dotsc, \beta_t \in \mathbb Z^n$.
        \item\label{lem:fg-fs-flange:3} There exists a flat-injective presentation
        \[ \varphi\colon \bigoplus_{j=1}^t \Sigma^{\beta_j} R \longrightarrow \bigoplus_{i = 1}^s \Sigma^{\alpha_i} E(\Bbbk) \]
        with $\alpha_1, \dotsc, \alpha_s, \beta_1, \dotsc, \beta_t \in \mathbb Z^n$, which admits a free-injective matrix.
    \end{enumerate}

    \begin{proof}
        \ref{lem:fg-fs-flange:1}:
        It suffices to show that $M$ is Artinian.
        For a submodule $U \subseteq M$ let $d_U\coloneqq \sum_{a \in \supp M} \dim_{\Bbbk}(M_a) = \dim_{\Bbbk} (\Bbbk \otimes_R M) < \infty$.
        For any chain
        \[ M \supseteq U_1 \supseteq U_2 \supseteq \dotsb \supseteq U_m \supseteq U_{m+1} \supseteq \dotsb \]
        of submodules of $M$ we have a non-increasing sequence $(d_{U_m})_{m \in \mathbb N}$ of non-negative integers.
        This sequence stabilizes, that is, there exists $m_0 \in \mathbb N$ with $d_{U_m} = d_{U_{m_0}}$ for all $m \geq m_0$.
        Thus, we have $U_m = U_{m_0}$ for $m \geq m_0$.
        Hence, $M$ is Artinian, and it follows from a classical result, see e.g.\ \cite[Proposition 18.1.2]{ChristensenFoxbyHolm2024}, that $M$ can be embedded in $\bigoplus_{i=1}^s \Sigma^{\alpha_i} E(\Bbbk)$.

        \ref{lem:fg-fs-flange:2}:
        Choose the uniquely determined map $\varepsilon\colon \bigoplus_{j=1}^t \Sigma^{\beta_j} R \to M$ with $\varepsilon(\mathbf e_j) = \mathbf g_j$ for generators $\mathbf g_1, \dotsc, \mathbf g_j$ of $M$.

        \ref{lem:fg-fs-flange:3}:
        We choose, by using \ref{lem:fg-fs-flange:1}, a monomorphism $\hat \varphi\colon M \hookrightarrow \bigoplus_{i=1}^s \Sigma^{\alpha_i} E(\Bbbk)$ and, by using \ref{lem:fg-fs-flange:2}, an epimorphism $\check \varphi\colon \bigoplus_{j=1}^t \Sigma^{\beta_j} R \twoheadrightarrow M$.
        Then $\varphi \coloneqq \hat \varphi \circ \check \varphi$ is a flat-injective presentation of $M$ since we have $\im \varphi \cong \im {\check\varphi} \cong M$.
    \end{proof}
\end{lemma}

Therefore, the structure of $E(\Bbbk)$ is crucial to handle the generators of a module.
Since every module is assumed to be finitely supported, we may assume further that $\supp M \subseteq \mathbb N^n$ for every $R$-module $M$; in general, $\Sigma^{\min(\supp M)} M$ is supported in $\mathbb N^n$.

Now we apply the Čech complex, also known as the infinite dual Koszul complex, see \cite[pp.\ 567--569]{ChristensenFoxbyHolm2024}, to obtain the analytic structure of $E(\Bbbk)$, and by that the analytic structure of the generators of a module $M$ given a flat-injective presentation $\varphi\colon F \to M$ of $M$.

We write $\varepsilon_i$ for the $i$-th standard lattice generator in $\mathbb Z^n$.
That is, $\varepsilon_i = (\varepsilon_{i1}, \dotsc, \varepsilon_{in})$ with $\varepsilon_{ii} = 1$ and $\varepsilon_{ij} = 0$ for $1 \leq j \leq n$, $i \neq j$.

\begin{definition}[Čech complex]
    Let $n \in \mathbb N$, and write $[n]\coloneqq \{ 1, \dotsc, n \}$ and $\binom{[n]}k\coloneqq \{ \sigma \in [n] \mid \lvert \sigma \rvert = k \}$.
    For $k \in \mathbb N$ and $\sigma \in \binom{[n]}k$ let $(X_\sigma)$ denote the ideal generated by $\{ X_l \mid l \in \sigma \}$ and let $\varepsilon_\sigma = \sum_{l \in \sigma} \varepsilon_l \in \mathbb Z^n$.
    The \emph{$n$-dimensional (graded) Čech complex} is the complex $\check C_\bullet^R(X_1, \dotsc, X_n)$
    with
    \[ \check C_k^R(X_1, \dotsc, X_n) \coloneqq \bigoplus_{\sigma \in \binom{[n]}k} \Sigma^{\varepsilon_\sigma} R_{(X_\sigma)} \]
    and differentials $\kappa_k\colon \check C_k^R(X_1, \dotsc, X_n) \to \check C_{k-1}^R(X_1, \dotsc, X_n)$ defined by
    \[ \kappa_k \coloneqq \sum_{\sigma = \{ q_1 < \dotsb < q_k \}} \sum_{l = 1}^k (-1)^{kl} \iota_{\sigma, q_l} \]
    where $\iota_{\sigma, q_l}$ is the canonical localization map $\Sigma^{\varepsilon_{\sigma}} R_{(X_\sigma)} \to \Sigma^{\varepsilon_{\sigma \setminus \{ q_l \}}} R_{(X_{\sigma \setminus \{ q_l \}})}$.
\end{definition}

\begin{example}
    For $n = 2$ the Čech complex $\check C^R_\bullet(X_1, X_2)$ is given by the complex
    \[ \begin{tikzcd}[ampersand replacement=\&]
        0 \rar \&
        \Sigma^{(1, 1)} R_{(X_1, X_2)} \rar["{\left(\begin{smallmatrix} 1 \\ 1 \end{smallmatrix}\right)}"] \&
        \Sigma^{(1, 0)} R_{(X_1)} \oplus \Sigma^{(0, 1)} R_{(X_2)} \rar["{\left(\begin{smallmatrix} -1 & 1 \end{smallmatrix}\right)}"] \&
        R_{(0)} \rar \&
        0 \rlap{\text.}
    \end{tikzcd} \]
    It can be seen that this complex is acyclic with $H_0(\check C^R_\bullet(X_1, X_2)) \cong E(\Bbbk) \cong \Bbbk[X_1^{-1}, X_2^{-1}]$, for instance by a dimension counting argument as in \cite[Lemma 3.8]{lenzen2024computing}.
\end{example}

Let $E = \bigoplus_{i = 1}^s \Sigma^{\alpha_i} E(K)$ be an injective module.
We consider the monomial submodule $U_E$ generated by
\[ G_E \coloneqq \big\{ X^{(\alpha_{ik} + 1) \cdot \varepsilon_k} \mathbf e_i \mid i \in \{ 1, \dotsc, s \},\ k \in \{ 1, \dotsc, n \} \big\} \subseteq R^s \text, \]
where we write $\alpha_i = (\alpha_{i1}, \dotsc, \alpha_{in}) \in \mathbb Z^n$.

\begin{proposition}%
    \label{prop:inj-pos-graded}
    For any injective module $E = \Sigma^{\alpha} E(\Bbbk)$ there is a graded isomorphism
    \[ E_{\mathbb N^n} \cong R/G_E \text. \]

    \begin{proof}
        We consider the Čech complex $\check C_\bullet \coloneqq \check C^R_\bullet(X_1, \dotsc, X_n)$.
        Then we have $H_0(\check C_\bullet) = E(\Bbbk)$, cf.\ \cite[Proposition 13.1.4]{ChristensenFoxbyHolm2024}.
        Therefore, $E(\Bbbk)$ is isomorphic to $R_{(0)} / (\sum_{k=1}^n \Sigma^{\varepsilon_k} R_{(X_k)})$ where we identify $\Sigma^{\varepsilon_k} R_{(X_k)}$ with its homomorphic image in $R_{(0)}$ under localization.
        This identification leads to the submodule
        \[ \sum_{k=1}^n \Sigma^{\varepsilon_k} R_{(X_k)} = \sum_{k=1}^n \langle X_k \rangle_{R_{(0)}} \text, \]
        which is generated by the set $\{ X_1, \dotsc, X_k \} \subseteq R_{(0)}$.
        Applying a shift by $\alpha \in \mathbb Z^n$ and taking the $\mathbb N^n$-graded part gives
        \[ (\Sigma^{\alpha} E(\Bbbk))_{\mathbb N^n} = (\Sigma^{\alpha} R_{(0)})_{\mathbb N^n} / (X_1^{\alpha_1+1}, \dotsb, X_n^{\alpha_n + 1}) \text. \]
        Since we can identify $(\Sigma^{\alpha} R_{(0)})_{\mathbb N^n} \cong (R_{(0)})_{\mathbb N^n}$ with the module $R$, we have a graded isomorphism
        \[ (\Sigma^{\alpha} E(k))_{\mathbb N^n} \cong R / (X_1^{\alpha_1 + 1} \dotsm X_n^{\alpha_n + 1}) \text. \]
    \end{proof}
\end{proposition}

\begin{remark}%
    \label{rem:E as Rd mod U}
    For $E \coloneqq \bigoplus_{i = 1}^s \Sigma^{\alpha_i} E(\Bbbk)$, because of Prop.~\ref{prop:inj-pos-graded} we have
    \[ E_{\mathbb N^n} \cong \bigoplus_{i = 1}^s R/G_{\Sigma^{\alpha_i} E(\Bbbk)} \cong R^s / U \]
    with $U \coloneqq  \bigoplus_{i = 1}^s G_{\Sigma^{\alpha_i} E(\Bbbk)}$.
    Hence, the $\mathbb N^n$-graded part of any cofree module can be interpreted as a quotient of $R^d$ for some $d \in \mathbb N$.
\end{remark}

If $M$ is a graded $\Bbbk[X_1, \dotsc, X_n]$-module that is finitely generated and finitely supported, we may assume that $M_{\mathbb N^n}$ is isomorphic to $M$.
For if not, let $g_1, \dotsc, g_t$ be homogeneous generators of $M$ with degrees $\beta_1, \dotsc, \beta_t \in \mathbb Z^n$, respectively.
Then for $\gamma = \beta_1 \wedge \dotsb \wedge \beta_s$ the shift functor $\Sigma^\gamma$ is an autoequivalence of categories, which sends $M$ to the module $\Sigma^\gamma M$ isomorphic to $(\Sigma^\gamma M)_{\mathbb N^n}$.
Thus, the shifted module $\Sigma^\gamma M$ satisfies the same properties as $M$.
Especially, if $\varphi\colon F \to E$ is a (minimal) flat-injective presentation of $M$, the map $\Sigma^\gamma \varphi\colon \Sigma^\gamma F \to \Sigma^\gamma E$ is a (minimal) flat-injective presentation of $M$.

\begin{lemma}
    The set $G_E$ is a Gröbner basis with respect to any monomial order.

    \begin{proof}
        For $i, j \in \{ 1, \dotsc, s \}$ and $k, l \in \{ 1, \dotsc, n \}$ the S-polynomial of $X^{(\alpha_i + 1) \varepsilon_k} \mathbf e_i$ and $X^{(\alpha_j + 1) \varepsilon_l} \mathbf e_j$ is zero.
        Hence, $G_E$ is a Gröbner basis of its generated submodule.
    \end{proof}
\end{lemma}

For a free-injective matrix $A = (a_{ij})_{i, j}$ with cogenerator degrees $\alpha_1, \dotsc, \alpha_s$ and generator degrees $\beta_1, \dotsc, \beta_t$ we define
\[ H_A \coloneqq \bigg\{ X^{\beta_j} \sum_{i = 1}^s a_{ij} \mathbf e_i \mid j \in \{ 1, \dotsc, t \} \bigg\} \subseteq R^s \text. \]
Further, we define $V_A \coloneqq \langle H_A \rangle + U_E$ with $E \coloneqq \bigoplus_{i=1}^s \Sigma^{\alpha_i} E(\Bbbk)$.

\begin{definition}%
    \label{def:matrix-Groebner}
    Let $\alpha_1, \dotsc, \alpha_s$ and $\beta_1, \dotsc, \beta_t$ in $\mathbb N^n$.
    A free-injective matrix $A$ with cogenerator degrees $\alpha_1, \dotsc, \alpha_s$ and generator degrees $\beta_1, \dotsc, \beta_t$ is said to be in \emph{(reduced) Gröbner form} if the set $H_A$ is a (reduced) Gröbner basis relative to $U_E$.
\end{definition}

We immediately obtain a characterization of graded matrices in Gröbner form by applying condition~\ref{RelGBCompletion(2)} of Proposition~\ref{prop:RelGBCompletionToClassicalGB}.

\begin{corollary}
    Let $H = \{ \mathbf h_1, \dotsc, \mathbf h_t \} \subseteq R^s$.
    The set $H$ is a Gröbner basis of $\langle H \rangle + \langle G_E \rangle$ relative to $\langle G_E \rangle$ if and only if
    \begin{enumerate}[label=(\roman*)]
        \item $\mathbf S(\mathbf h_j, X^{(\alpha_{ik} + 1) \cdot \varepsilon_k} \mathbf e_i)$ is divisible by $G_E \cup H_A$ for all $i \in \{ 1, \dotsc, s \}$, $j \in \{ 1, \dotsc, t \}$, and $k \in \{ 1, \dotsc, n \}$, and
        \item $\mathbf S(\mathbf h_j, \mathbf h_l)$ is divisible by $G_E \cup H_A$ for all $1 \leq j < l \leq t$.
    \end{enumerate}
\end{corollary}

\begin{remark}%
    \label{rem:s-pol-flange}
    From the previous corollary it immediately follows that a free-injective matrix $A$ as in Definition~\ref{def:matrix-Groebner} is in Gröbner form if and only if
    \begin{enumerate}[label=(\roman*)]
        \item the S-polynomials%
        \begin{equation}%
            \label{eq:s-pol-flange1}
            \mathbf S(X^{\beta_j} \sum_{i' = 1}^s a_{i'j} \mathbf e_{i'}, X^{(\alpha_{lk} + 1) \cdot \varepsilon_k} \mathbf e_i) = X^{(\alpha_{lk} + 1) \cdot \varepsilon_k \vee \beta_j} \sum_{i' \neq l} a_{i'j} \mathbf e_{i'}
        \end{equation}
        for $1 \leq i \leq s$, $1 \leq j \leq t$ and $1 \leq k \leq n$ are divisible by $G_E \cup H_A$ and
        \item the S-polynomials
        \begin{equation}%
            \label{eq:s-pol-flange2}
            \mathbf S(X^{\beta_j} \sum_{i = 1}^s a_{ij} \mathbf e_i, X^{\beta_l} \sum_{i = 1}^s a_{il} \mathbf e_i)
        \end{equation}
        for $1 \leq j < l \leq t$ are divisible by $G_E \cup H_A$.
    \end{enumerate}
\end{remark}

This allows us to give a simplified variant of Buchberger's algorithm for free-injective matrices, Algorithm~\ref{alg:BuchbergerFlange}, which makes use of the aforementioned criterion.
In contrast to the general case, we need not consider the S-polynomials of pairs inside $G_E$ since $G_E$ is monomial.

Furthermore, it suffices to divide the S-polynomial \eqref{eq:s-pol-flange1} by the set $H$ after writing the S-polynomial in normal form with respect to $G_E$.
This normal form is determined by removing every summand $a_{i'j}\mathbf e_{i'}$ of the right side of \eqref{eq:s-pol-flange1} for that $(\alpha_i + \varepsilon_k) \vee \beta_j \leq \alpha_{i'}$ does not hold.

For a free-injective matrix $A$, we write $A_j$ for the $j$-th column of $A$, regarded as a free-injective $s \times 1$-matrix with generator degree $\beta_j$ via monomialization, see Section~\ref{ssect:monomialization}.
Further, for a free-injective $s \times 1$-matrix $\mathbf v$, we write $\mdeg \mathbf v$ for the unique generator degree of $\mathbf v$.

We assume that every free-injective matrix, also called \emph{free-injective vectors}, is implicitly converted to normal form, which means that for any free-injective matrix $A = (a_{ij})_{i, j}$, $a_{ij} \neq 0$ only if $\alpha_i \leq \beta_j$.
That is, we regard free-injective matrices as a first-class data structure.

Lastly, for $\beta = (\beta_1, \dotsc, \beta_n)$ and $\beta' = (\beta'_1, \dotsc, \beta'_n)$ we define the \emph{join of $\beta$ and $\beta'$} to be $\beta \vee \beta' \coloneqq (\max(\beta_1, \beta'_1), \dotsc, \max(\beta_n, \beta'_n))$.

\begin{algorithm}
    \caption{Polynomial division algorithm for free-injective vectors}
    \label{alg:MonomialDivision}
    \begin{algorithmic}[1]
        \Statex \textbf{Input:} A free-injective vector $\mathbf v = [\mathbf v_1, \dotsc,  \mathbf v_s]^T$, a free-injective $(s \times t)$-matrix $A$
        \Statex \textbf{Output:} Remainder of $\mathbf v$ modulo $\{ A_1, \dotsc, A_t \}$
        \Function{MonomialDivision}{$\mathbf v, A$}
            \State \Return $\textsc{RelativeDivision}(G_E, \{ A_1, \dotsc, A_t \}, X^{\mdeg \mathbf v} \sum_{i = 1}^s \mathbf v_i \mathbf e_i)$
        \EndFunction
    \end{algorithmic}
\end{algorithm}

\begin{algorithm}
    \caption{Relative Buchberger's algorithm for free-injective matrices}
    \label{alg:BuchbergerFlange}
    \begin{algorithmic}[1]
        \Statex \textbf{Input:} A free-injective matrix $A$ with cogenerator degrees $\alpha_1, \dotsc, \alpha_s$ and generator degrees $\beta_1, \dotsc, \beta_t$.
        \Statex \textbf{Output:} A free-injective matrix $A'$ in Gröbner form with $\im A = \im A'$.
        \Function{BuchbergerFlange}{$A$}
            \State $A' \leftarrow A \in \Bbbk^{s \times t}$
            \State $Q \leftarrow \{ (A_i, A_j) \mid 1 \leq i < j \leq t \} \cup \{ (A_j, X^{(\alpha_{ik} + 1) \cdot \varepsilon_k} \mathbf e_i) \mid i \in \{ 1, \dotsc, s \}, j \in \{ 1, \dotsc, t \}, k \in \{ 1, \dotsc, n \} \}$
            \While{$Q \neq \emptyset$}
                \State Choose $(\mathbf v, \mathbf w) \in Q$
                \State $Q \leftarrow Q \setminus \{ (\mathbf v, \mathbf w) \}$
                \State $(\mathbf q, \dotsc) \coloneqq \textsc{MonomialDivision}(\mathbf S(\mathbf v, \mathbf w), A')$
                \If{$\mathbf q \neq 0$}
                    \State $Q \leftarrow Q \cup \{ (\mathbf q, \mathbf v) \mid \text{each column $\mathbf v$ in $A'$}\} \cup \{ (\mathbf q, X^{(\alpha_{ik}+1)\cdot\varepsilon_k} \mathbf e_i) \mid i \in \{ 1, \dotsc, s \}, k \in \{ 1, \dotsc, n \} \}$
                    \State $A' \leftarrow \begin{bmatrix} A' & \mathbf q \end{bmatrix}$
                \EndIf
            \EndWhile
            \State \Return $A'$
        \EndFunction
    \end{algorithmic}
\end{algorithm}

\begin{theorem}
    Algorithm~\ref{alg:BuchbergerFlange} terminates and given a free-injective matrix $A$ it computes a free-injective matrix $A'$ in Gröbner form with $\im A = \im A'$.
    In particular, $A$ is a submatrix of $A'$.

    Hence, every finitely generated and finitely supported $n$-graded $R$-module $M$ admits a representation as a free-injective matrix $A$ in Gröbner form.

    \begin{proof}[Proof]
        In the following we identify the free-injective matrix $A'$ with its set of columns.

        Termination of the algorithm is equivalent to termination of Buchberger's algorithm in its non-relative variant for the input $A' \cup G_E$ where $E = \bigoplus_{i = 1}^s \Sigma^{\alpha_i} E(\Bbbk)$.
        Therefore, the algorithm terminates.

        To show correctness, we write $A^{\prime(i)}$ to denote the free-injective matrix $A'$ after the $i$-th iteration of the outer loop in Algorithm~\ref{alg:BuchbergerFlange}, lines 4--12.
        In particular, we have $A^{\prime(0)} = A$.
        Since the algorithm terminates for any input, the sequence $(A^{\prime(i)})_{i \geq 0}$ eventually becomes stationary.
        Thus, we can choose some $i_0 \geq 0$ with $A^{\prime(i)} = A^{\prime(i_0)}$ for all $i \geq i_0$.

        To show that $\im A$ coincides with the output of the algorithm, it suffices by induction that $\im A^{\prime(i)} = \im A^{\prime(i+1)}$ for every $i \geq 0$.
        Let $i \in \mathbb N$ and $(\mathbf v, \mathbf w) \in Q$ in the $(i+1)$-th iteration.
        Then whenever $\mathbf q \neq 0$ obtained from $(\mathbf v, \mathbf w)$ in line 7 of Algorithm~\ref{alg:BuchbergerFlange}, the element $\mathbf q$ is by construction an element of $\im A^{\prime(i)}$.
        Thus, the invariant $\im A^{\prime(i+1)} = \im A^{\prime(i)} + R\mathbf q = \im A^{\prime(i)}$ is satisfied.
        By induction, we obtain $\im A = \im A^{\prime(i)} = \im A^{\prime(i_0)}$ for any $i \geq 0$.

        Lastly, every matrix $A^{\prime(i+1)}$, for $i \geq 0$, is constructed in line 10 of Algorithm~\ref{alg:BuchbergerFlange} such that $A^{\prime(i)}$ is a submatrix of $A^{\prime(i+1)}$.
        Inductively, we see that the output matrix $A'$ of the algorithm contains $A$ as a submatrix.
    \end{proof}
\end{theorem}

\begin{example}%
    \label{ex:flange-4}
    We reconsider the free-injective matrix $A$ from Example~\ref{ex:flange-1} and endow $R^6$ with the POT-extended graded lexicographic order subject to $\mathbf e_1 \prec \mathbf e_2 \prec \dotsb$.
    By applying Algorithm~\ref{alg:FreePresentation} to $A$ we obtain the free-injective matrix
    \[ \begin{tikzpicture}[ampersand replacement=\&, baseline=(m-4-1.base)]
    	\matrix (m) [matrix of math nodes] {
    		\&[1em] (1,0) \& (0,1) \& (2,0) \& (1,1) \& (1,1) \& (2,1) \& (1,1) \\
    		(1,0) \& 1 \& 0 \& 0 \& 0 \& 0 \& 0 \& 0 \\
    		(0,1) \& 0 \& 1 \& 0 \& 0 \& 0 \& 0 \& 0 \\
    		(2,0) \& 1 \& 0 \& 1 \& 0 \& 0 \& 0 \& 0 \\
    		(1,1) \& 1 \& 0 \& 0 \& 1 \& 0 \& 0 \& 1 \\
    		(1,1) \& 0 \& 1 \& 0 \& 0 \& 1 \& 0 \& -1 \\
    		(2,1) \& 1 \& 1 \& 1 \& 1 \& 1 \& 1 \& 0 \\
    	};
    	\node[rectangle, left delimiter={[}, right delimiter={]}, fit=(m-2-2.north west) (m-7-8.south east)] {};
    \end{tikzpicture} \]
    in Gröbner form.
\end{example}

\begin{algorithm}
    \caption{Relative Schreyer's algorithm for free-injective matrices}
    \label{alg:FreePresentation}
    \begin{algorithmic}[1]
        \Statex \textbf{Input:} A free-injective matrix $A$ in Gröbner form with cogenerator degrees $\alpha_1, \dotsc, \alpha_s$ and generator degrees $\beta_1, \dotsc, \beta_t$.
        \Statex \textbf{Output:} A free presentation $F$ of $\im A$.
        \Function{FreePresentation}{$A$}
            \State $F \leftarrow [] \in \Bbbk^{t \times 0}$
            \For{$1 \leq j_1 < j_2 \leq t$ with $\lm(A_{j_1}) = X^{\alpha} \mathbf e_i$ and $\lm(A_{j_2}) = X^{\beta} \mathbf e_i$ for some $i \in \mathbb N^n$}
                \State $(\mathbf q, a_1, \dotsc, a_t) \coloneqq \textsc{MonomialDivision}(\mathbf S(A_{j_1}, A_{j_2}), A)$
                \State $\sigma \coloneqq X^{\beta_{j_1} \vee \beta_{j_2} - \beta_{j_1}} \mathbf e_{j_1} - X^{\beta_{j_1} \vee \beta_{j_2} - \beta_{j_2}} \mathbf e_{j_2} - \sum_{j_3 = 1}^t a_{j_3} \mathbf e_{j_3} \in R^t$
                \If{$\sigma \neq 0$}
                    \State $F \leftarrow \begin{bmatrix}
                        F & \sigma
                    \end{bmatrix}$
                \EndIf
            \EndFor
            \For{$1 \leq j \leq t$, $1 \leq i \leq s$, and $k \in \{ 1, \dotsc, n \}$ with $\lm(A_j) = X^{\alpha} \mathbf e_i$}
                \State $(\mathbf q, a_1, \dotsc, a_t) \coloneqq \textsc{MonomialDivision}(\mathbf S(A_j, X^{(\alpha_{ik}+1)\cdot\varepsilon_k}\cdot \mathbf e_i), A)$
                \State $\sigma \coloneqq X^{\beta_j \vee (\alpha_{ik} + 1) \varepsilon_k - \beta_j} \mathbf e_j - \sum_{l = 1}^t a_l \mathbf e_l \in R^t$
                \If{$\sigma \neq 0$}
                    \State $F \leftarrow \begin{bmatrix}
                        F & \sigma
                    \end{bmatrix}$
                \EndIf
            \EndFor
            \State \Return $F$
        \EndFunction
    \end{algorithmic}
\end{algorithm}

\FlangeToFree*

\begin{proof}[Proof of Theorem~\ref{thm:free presentation of a free-cofree presentation}]
    Termination of the procedure \textsc{FreePresentation} is obvious.
    Furthermore, every column of the output matrix $F$ of the algorithm is a syzygy of the module $\im A$ by construction.
    Conversely, every syzygy of $\im A$ is a linear combination of columns of $F$ by Theorem~\ref{thm:KernelOfPresentation:intro}.
\end{proof}

\begin{example}%
    \label{ex:flange-5}
    Recall the minimal flat-injective matrix $\tilde A$ for the $\Bbbk[X_1, X_2]$-module $M$ from Example~\ref{ex:flange-1}.
    Algorithm~\ref{alg:BuchbergerFlange} gives the flat-injective matrix
    \[ \begin{tikzpicture}[ampersand replacement=\&, baseline=(m-3-1.north)]
    	\matrix (m) [matrix of math nodes] {
    		\&[1em] (1,0) \& (0,1) \& (1,1) \\
    		(1,1) \& 1 \& 0 \& 1 \\
    		(2,1) \& 1 \& 1 \& 0 \\
    	};
    	\node[rectangle, left delimiter={[}, right delimiter={]}, fit=(m-2-2.north west) (m-3-4.south east)] {};
    \end{tikzpicture} \]
    in Gröbner form, where $R^2$ is endowed with the POT-extension of the graded lexicographic order subject to $\mathbf e_1 \prec \mathbf e_2$.
    Then, by application of Algorithm~\ref{alg:FreePresentation}, we obtain the free presentation $\partial_1:F_1\to F_0$ of the module $M$ that is given by
    \[ \partial_1 = \begin{bmatrix}
    	X_2 & X_2^2 & -X_1X_2 & 0 & 0 & 0 & X_1^2 & 0 & 0 \\
    	-X_1 & 0 & X_1^2 & X_2 & 0 & 0 & 0 & X_1^3 & 0 \\
    	-1 & 0 & 0 & 0 & X_1 & X_2 & 0 & 0 & X_1^2
    \end{bmatrix} \text. \]
    Notice that it can directly be read off from the free presentation $\partial_1$ that the third column of the flat-injective matrix in Gröbner form is redundant.
\end{example}

\begin{remark}
    Since the application of Algorithm~\ref{alg:FreePresentation} requires a free-injective matrix in Gröbner form as input, it is customary to apply Algorithm~\ref{alg:BuchbergerFlange} beforehand to a free-injective matrix.
    Both algorithms depend on the computation of a monomial division of S-polynomials, compare line~7 in Algorithm~\ref{alg:BuchbergerFlange} and lines~4 and~11 in Algorithm~\ref{alg:FreePresentation}, which are carried out on similar elements.

    Therefore, an efficient implementation of the concatenation of both algorithms can avoid performing duplicate instances of the monomial division algorithm by storing the coefficients $a_1, \dotsc, a_t$ needed in lines~5--8 and~12--15 of Algorithm~\ref{alg:FreePresentation} during the computation of a free-injective matrix in Gröbner form.
\end{remark}

\section{Computing minimal multigraded resolutions}%
\label{sect:minimal-res}
In this section, we describe a standard technique in multigraded computational algebra,
commonly referred to as \emph{pruning}, for reducing a given generating set
arising from a free presentation of a module.
More precisely, suppose we are given a free presentation
\[
F_1 \xrightarrow{\partial_1} F_0 \xrightarrow{\varepsilon} M \longrightarrow 0,
\]
of a finitely generated $n$-graded module~$M$, where $\partial_1$ is a matrix representing the map, whose rows correspond to generators and whose columns encode relations (first-order syzygies).
Pruning minimization provides a systematic procedure for eliminating redundant generators---corresponding to rows of $\partial_1$---thereby producing a \emph{generator-minimal} presentation of $M$.
This procedure can be applied iteratively in the context of a free resolution
\[
\cdots \longrightarrow F_2 \xrightarrow{\partial_2} F_1 \xrightarrow{\partial_1} F_0 \xrightarrow{\varepsilon} M \longrightarrow 0.
\]
At each stage, the matrix $\partial_i$ represents the syzygy modules of the preceding maps.
By performing pruning minimization on each $\partial_i$, one eliminates redundant generators at every level.
In the $n$-graded setting, this yields a \emph{minimal $n$-graded free resolution} of $M$, since $n$-graded cancellations across different multidegrees are impossible.

Thus, a minimal resolution of a submodule $V/U \subseteq F/U$ can be obtained by first computing a presentation of $V/U$, then iteratively forming the syzygy modules, and finally applying pruning minimization to each matrix in the resulting resolution.
As minimization of resolutions can be obtained by generator-minimizing $V/U\subseteq F/U$ and each $i$-th syzygy module, $\ker(\partial_i)\subseteq F_i$, for $i=1,\ldots,n$, it suffices to show how to generator-minimize a given module when given as $V/U$.
Generator-minimization is also useful for generator-minimizing free-cofree presentations.

In the first subsection, we recall the pruning minimization procedure for generators of a subquotient $V/U \subseteq F/U$, and in the second subsection, we show how to generator-minimize a given free–cofree presentation by applying pruning. In the third subsection, we present an alternative, linear algebraic, viewpoint of generator minimization.

\subsection{Generator-minimization of free presentations}
In this subsection, we describe in detail the procedure of computing minimal generators
of an $R$-submodule $V/U$ of $F/U$ from a given presentation of $V/U$. More precisely, we start with a
generating set $H$ of $V$ (relative to $U$) together with a generating set of the syzygy module
$\Syz(H)$, for instance computed via Schreyer's theorem from a Gröbner basis of $H$.
From this data, we construct a minimal graded generating set $H'$ for the quotient module $V/U$.

The generator-minimization problem cannot be solved by Gröbner basis reduction alone.
Namely, a reduced Gröbner basis of $V$ relative to $U$ need not yield a minimal generating
set of the quotient $V/U$. This phenomenon is illustrated in Example~\ref{ex:non-minimal} below.

\begin{example}%
    \label{ex:non-minimal}
    Consider the submodule $U \subseteq R^2$ generated by
    \begin{align*}
        \mathbf g_1 &\coloneqq X_1 \mathbf e_1 - X_1 \mathbf e_2 \text, \\
        \mathbf g_2 &\coloneqq X_2 \mathbf e_1 + X_2 \mathbf e_2 \text.
    \end{align*}
    The S-polynomial of $\mathbf g_1, \mathbf g_2$ is given by $\mathbf S(\mathbf g_1, \mathbf g_2) = -2X_1 X_2 \mathbf e_2$,
    which is in normal form w.r.t.\ $\{ \mathbf g_1, \mathbf g_2 \}$.
    Let $\mathbf g_3 = X_1 X_2 \mathbf e_2$.
    A Gröbner basis of $U$ is given by $G = \{ \mathbf g_1, \mathbf g_2, \mathbf g_3 \}$.

    Further, the Gröbner basis $G$ is reduced and minimal, but does not minimally generate $G$ since the syzygy
    \[ X_2 \mathbf g_1 - X_1 \mathbf g_2 + 2 \mathbf g_3 = 0 \]
    does not have coefficients in $\mathfrak m = (X_1, \dotsc, X_n)$.

    Hence, reduced $n$-graded Gröbner bases are not necessarily minimal generating sets.
\end{example}

Gröbner basis reduction removes redundancies detected at the level of leading terms, but it does not necessarily identify generators that become redundant after passing to the quotient. Consequently, pruning minimization is required to eliminate superfluous generators.

Below, we adapt the standard generator-minimization (pruning) of matrices associated to multigraded presentations \cite[Defn.~1.24, p.~12]{miller2005combinatorial} to the setting of subquotients $V/U\subseteq F/U$.

\begin{lemma}[Pruning minimization of generators, relative version]%
    \label{lem:Constant_terms_minimality_criterion}
    Let $U\subseteq V\subseteq R^d$ be graded submodules, and let $H=\{\mathbf{h}_1+U,\dotsc,\mathbf{h}_t+U\}$ be a graded generating set of $V/U$.
    Let $S\subseteq R^t$ be a graded generating set of $\Syz(H)$.
    Let the matrix $A\in R^{t\times |S|}$ have the elements of $S$ as columns.
    Then
    \begin{enumerate}
        \item\label{lem:Constant_terms_minimality_criterion:1} $H$ minimally generates $V/U$ if and only if there is no constant module term $\lambda \mathbf{e}_i$, $1\leq i\leq t$, $\lambda\in \Bbbk\setminus\{0\}$, in $\bigcup_{\mathbf{s}\in S}\supp(\mathbf{s})$.
        \item\label{lem:Constant_terms_minimality_criterion:2} $H$ minimally generates $V/U$ if and only if no non-zero constant appears as an entry of $A$.
        \item\label{lem:Constant_terms_minimality_criterion:3} Let $A_0$ be the matrix obtained from $A$ by taking the constant term of each entry.
        Let $r=\rk(A_0)$ and let $I\times J$ with $|I|=|J|=r$ be the index set of an invertible square submatrix of $A_0$.
        Then $H_{\min}\coloneqq\{\mathbf{h}_i+U\in H\mid i\notin I\}$ minimally generates $V/U$.
        \item\label{lem:Constant_terms_minimality_criterion:4} Fix a monomial ordering on $R$ and extend it to a POT order on $R^t$. Given $A$ as input, there is an algorithm to compute a minimal generating set $H_{\min}$ of $V/U$.
     \end{enumerate}
\end{lemma}
\begin{proof}
    For assertion~\ref{lem:Constant_terms_minimality_criterion:1}, $H$ is not minimal if and only if there is a syzygy of $H$ with a constant module term $\lambda \mathbf{e}_i$ in its support, if and only if every generating set of $\Syz(H)$ contains an element with a constant module term $\lambda \mathbf{e}_i$ in its support.

    Assertion \ref{lem:Constant_terms_minimality_criterion:2} is immediate from \ref{lem:Constant_terms_minimality_criterion:1}.

    For assertion~\ref{lem:Constant_terms_minimality_criterion:3}, let $I=\{i_1,\dotsc,i_r\}$ and $J=\{j_1,\dotsc,j_r\}$. Say the column $\mathbf{s}_{j_1}$ has a non-zero constant entry in the component indexed by $i_1$ (otherwise, renumber $I$). For $j\neq j_1$, modify the $j$-th column of $A$ by adding a suitable multiple of $\mathbf{s}_{j_1}$ such that the only non-zero entry in the $i_1$-th row that remains is in $\mathbf{s}_{j_1}$. Remove the $i_1$-th row and the $j_1$-th column to obtain a new matrix $A^{(1)}$, with row and column indices inherited from those of $S$.

    By a standard minimization result (e.g. \cite[Proof of Thm. 3.15]{CLOUsing}) for free resolutions applied to our free presentation, the subset $H^{(1)}=\{\mathbf{h}_i+U\in H\mid i\notin \{i_1\}\}$ of $H$ is a generating set of $V/U$ and the set of columns of $A^{(1)}$ is a generating set $\Syz(H^{(1)})$. By the assumptions on the index sets $I$ and $J$, the column indexed by $j_2$ in $A^{(1)}$ has a non-zero constant entry in position $i_2$ (otherwise, renumber $I$). Thus we can iterate the argument and in the end we obtain that $H^{(r)}=\{\mathbf{h}_i+U\in H\mid i\notin I\}$ is a minimal set of generators of $V/U$.

    Assertion~\ref{lem:Constant_terms_minimality_criterion:4} follows from \ref{lem:Constant_terms_minimality_criterion:3} and its proof. The POT order provides an explicit procedure to select the index sets $I$ and $J$ and the constant matrix entries that are used for the reductions in each step.
 \end{proof}

\begin{remark}\label{rem:minimizing_free_resolutions}
    Graded $R$-free resolutions can be minimized by an iterative application of Lemma~\ref{lem:Constant_terms_minimality_criterion}. For details see~\cite[Proof of Thm. 3.15]{CLOUsing}.
\end{remark}

\begin{example}\label{ex:Resolution}
    In this example, we apply the minimization of free resolutions to two different presentations of the same module, illustrating how a good choice of presentation can reduce intermediate expression growth.
    Let $R=\Bbbk[X_1,X_2]$ be the graded local polynomial ring in two variables.
    We consider the graded submodules $U\subseteq V\subseteq R^6$, where $U$ is generated by the columns of the matrix
    \[ \left[\begin{smallmatrix}
        X_1 X_2 & 0 & X_1^2 & 0 & 0 & 0 & 0 & 0 & 0 & 0 & 0 & 0 \\
        0 & X_1 X_2 & 0 & 0 & 0 & 0 & X_2^2 & 0 & 0 & 0 & 0 & 0 \\
        -X_1 X_2 & 0 & 0 & X_1^2 X_2 & 0 & 0 & 0 & X_1 X_2^2 & 0 & 0 & 0 & 0 \\
        0 & -X_1 X_2 & 0 & 0 & X_1^2 X_2 & 0 & 0 & 0 & X_1 X_2^2 & 0 & 0 & 0 \\
        0 & 0 & -X_1^2 & 0 & 0 & X_1^2 X_2 & 0 & 0 & 0 & X_1^3 & 0 & 0 \\
        0 & 0 & 0 & -X_1^2 X_2 & -X_1^2 X_2 & -X_1^2 X_2 & 0 & 0 & 0 & 0 & X_1^2 X_2^2 & X_1^3 X_2 \\
    \end{smallmatrix}\right] \text,
    \]
    and $V\coloneqq\operatorname{im}\partial_0$ where $\partial_0\colon F_0\to R^6$ is given by
    \[ \partial_0\coloneqq\begin{bmatrix}
        X_1 & 0 & 0 & 0 & 0 & 0 \\
        0 & X_2 & 0 & 0 & 0 & 0 \\
        0 & 0 & X_1 X_2 & 0 & 0 & 0 \\
        0 & 0 & 0 & X_1 X_2 & 0 & 0 \\
        0 & 0 & 0 & 0 & X_1^2 & 0 \\
        0 & 0 & 0 & 0 & 0 & X_1^2 X_2 \\
    \end{bmatrix}\, \text. \]
    Let $\varepsilon$ be the composition of $\partial_0$ with the canonical projection $V \twoheadrightarrow V/U$.
    Note that these two generating sets are even a reduced Gr\"{o}bner basis of $U$ and a reduced Gr\"{o}bner basis of $V$ relative to $U$, respectively.
    Then the (non-minimal) $R$-free resolution of $V/U$ induced by these Gr\"{o}bner bases,
    \[ \begin{tikzcd}[ampersand replacement=\&]
        0 \rar \& F_2 \rar["\partial_2"] \& F_1 \rar["\partial_1"] \& F_0 \rar["\varepsilon"] \& V/U \rar \& 0 \rlap{\text,}
    \end{tikzcd} \]
    with $\rank F_0=6$, $\rank F_1=12$, and $\rank F_2=6$,
    is described by matrices $\partial_1\in R^{6\times 12}$ and $\partial_2\in R^{12\times 6}$ that contain non-zero constants. We have
    \setcounter{MaxMatrixCols}{20}
    \[
    \partial_1= \begin{bmatrix}
        X_1 & X_2 & 0 & 0 & 0 & 0 & 0 & 0 & 0 & 0 & 0 & 0 \\
        0 & 0 & X_1 & X_2 & 0 & 0 & 0 & 0 & 0 & 0 & 0 & 0 \\
        0 & -1 & 0 & 0 & X_1 & X_2 & 0 & 0 & 0 & 0 & 0 & 0 \\
        0 & 0 & -1 & 0 & 0 & 0 & X_1 & X_2 & 0 & 0 & 0 & 0 \\
        -1 & 0 & 0 & 0 & 0 & 0 & 0 & 0 & X_1 & X_2 & 0 & 0 \\
        0 & 0 & 0 & 0 & -1 & 0 & -1 & 0 & 0 & -1 & X_1 & X_2 \\
    \end{bmatrix}
    \]
    and
    \[\partial_2 = \begin{bmatrix}
        X_2 & 0 & 0 & 0 & 0 & 0 \\
        -X_1 & 0 & 0 & 0 & 0 & 0 \\
        0 & X_2 & 0 & 0 & 0 & 0 \\
        0 & -X_1 & 0 & 0 & 0 & 0 \\
        -1 & 0 & X_2 & 0 & 0 & 0 \\
        0 & 0 & -X_1 & 0 & 0 & 0 \\
        0 & 0 & 0 & X_2 & 0 & 0 \\
        0 & 1 & 0 & -X_1 & 0 & 0 \\
        0 & 0 & 0 & 0 & X_2 & 0 \\
        1 & 0 & 0 & 0 & -X_1 & 0 \\
        0 & 0 & 0 & 0 & -1 & X_2 \\
        0 & 0 & 1 & 1 & 0 & -X_1 \\
    \end{bmatrix}.\]

    The minimization yields a minimal free resolution of $V/U$,
    \[ \begin{tikzcd}[ampersand replacement=\&]
        0 \rar \& \tilde F_2 \rar["\tilde \partial_2"] \& \tilde F_1 \rar["\tilde \partial_1"] \& \tilde F_0 \rar["\bar{\tilde \partial}_0"] \& V/U \rar \& 0 \rlap{\text,}
    \end{tikzcd} \]
    with $\rank \tilde{F}_0=2$, $\rank \tilde{F}_1=4$, and $\rank \tilde{F}_2=2$, where
    \begin{gather*}
        \bar{\tilde\partial}_0\coloneqq\begin{bmatrix}
            X_1 & 0 \\
            0 & X_2 \\
            0 & 0 \\
            0 & 0 \\
            0 & 0 \\
            0 & 0 \\
        \end{bmatrix}\text,\quad
        \tilde\partial_1\coloneqq\begin{bmatrix}
            0 & X_2^2 & -X_1 X_2 & X_1^2 \\
            X_2 & 0 & X_1^2 & 0 \\
        \end{bmatrix}\text,\\
        \tilde\partial_2\coloneqq\begin{bmatrix}
            -X_1^2 & 0 \\
            X_1 & -X_1^2 \\
            X_2 & 0 \\
            0 & X_2^2 \\
        \end{bmatrix}\text,
    \end{gather*}
    and $\bar{\tilde\partial}_0$ is the composition of $\tilde \partial_0$ with the canonical projection $V \twoheadrightarrow V/U$.

    Next we consider the homogeneous submodules $U' \subseteq V' \subseteq R^2$, where $U'$ is generated by the columns of the following matrix:
    \[ \begin{bmatrix}
        X_1^3 & X_1^2 X_2 &  0 & X_1 X_2^2 & 0 \\
        0 & -X_1^2X_2 & X_2^2 & 0 & X_1^3X_2
    \end{bmatrix} \text; \]
    and $V'$ is generated by the columns of the matrix:
    \[ \partial_0'=\begin{bmatrix}
        X_1 & 0 \\
        0 & X_2
    \end{bmatrix} \text. \]
    We consider the homomorphism $\varphi\colon V' \to V/U$ defined by
    \begin{align*}
        \varphi(\begin{bmatrix}
            X_1 & 0 & 0 & 0 & 0 & 0
        \end{bmatrix}^T) &= \begin{bmatrix}
            X_1 & 0
        \end{bmatrix}^T \text, \\
        \varphi(\begin{bmatrix}
            0 & X_2 & 0 & 0 & 0 & 0
        \end{bmatrix}^T) &= \begin{bmatrix}
            0 & X_2
        \end{bmatrix}^T \text, \\
        \varphi(\begin{bmatrix}
            0 & 0 & X_1 X_2 & 0 & 0 & 0
        \end{bmatrix}^T) &= \begin{bmatrix}
            X_1 X_2 & 0
        \end{bmatrix}^T \text, \\
        \varphi(\begin{bmatrix}
            0 & 0 & 0 & X_1 X_2 & 0 & 0
        \end{bmatrix}^T) &= \begin{bmatrix}
            0 & X_1 X_2
        \end{bmatrix}^T \text, \\
        \varphi(\begin{bmatrix}
            0 & 0 & 0 & 0 & X_1^2 & 0
        \end{bmatrix}^T) &= \begin{bmatrix}
            X_1^2 & 0
        \end{bmatrix}^T \text, \\
        \varphi(\begin{bmatrix}
            0 & 0 & 0 & 0 & 0 & X_1^2 X_2
        \end{bmatrix}^T) &= \begin{bmatrix}
            0 & X_1^2 X_2
        \end{bmatrix}^T \text.
    \end{align*}
    A simple calculation reveals that $\varphi$ is surjective with kernel given by $U'$.
    Thus, $\varphi$ induces an isomorphism $V'/U' \to V/U$.

    We note that the generating sets of $U'$ and $V'$ are a Gr\"{o}bner basis and a Gr\"{o}bner basis relative to $U$, respectively, e.g. for any position over term ordering $\prec$ with $\mathbf{e}_1\succ \mathbf{e}_2$.
    The $R$-free resolution of $V'/U'$ induced by these Gr\"{o}bner bases,
    \[ \begin{tikzcd}[ampersand replacement=\&]
        0 \rar \& F_2' \rar["\partial_2'"] \& F_1' \rar["\partial_1'"] \& F_0' \rar["\varepsilon'"] \& V'/U' \rar \& 0 \rlap{\text,}
    \end{tikzcd} \]
    is described by the differential matrices
    \[ \partial_1'=\begin{bmatrix}
       X_1^2 & X_1 X_2 & X_2^2 & 0 & 0 \\
       0 & -X_1^2 & 0 & X_2 & X_1^3
    \end{bmatrix} \text, \quad
    \partial_2'=\begin{bmatrix}
        X_2 & 0 & 0 \\
        -X_1 & X_2 & 0 \\
        0 & -X_1 & 0 \\
        0 & X_1^2 & X_1^3 \\
        -1 & 0 & -X_2
    \end{bmatrix}\, \text . \]
    Minimizing the resolution using the non-zero constant entry of $\partial_2'$ gives the following new maps $\tilde\partial_1', \tilde\partial_2'$, where the generating set of $V'$ is not changed:
    \[ \tilde\partial_1'=\begin{bmatrix}
       X_1^2 & X_1 X_2 & X_2^2 & 0 \\
       0 & -X_1^2 & 0 & X_2
    \end{bmatrix} \text, \quad
    \tilde\partial_2'=\begin{bmatrix}
        0 & -X_2^2 \\
        X_2 & X_1 X_2 \\
        -X_1 & 0 \\
        X_1^2 & X_1^3
    \end{bmatrix}\text. \]
    Thus the Betti numbers are $2,4,2$ (which we could of course also have inferred from our computation of the Betti numbers of $V/U$ and the isomorphism $\varphi$).
\end{example}

\subsection{Generator-minimization of free-cofree presentations}%
\label{sec:minimal free cofree presentation from reduced GBs}

In this subsection, we illustrate with an example how $n$-graded pruning can be used in practice to reduce a given free-cofree presentation of $M$ from a given (minimal) free presentation of $M$, and then we propose a 
 method for generator-minimizing a free-cofree presentation $\varphi\colon F\to E$ by applying pruning on the associated relative Schreyer's presentation of $\im\varphi\subseteq E$, where we express them as $\im\varphi\eqqcolon V/U$ and $E\eqqcolon F'/U$.
 We can then also cogenerator-minimize the given free-cofree presentation by applying Matlis duality (corresponding to transposing the free-cofree matrix) and then repeat the generator-minimization procedure.

\begin{example}%
    \label{ex:flange-6}
    We reconsider the module~$M$ from Example~\ref{ex:flange-1} with the free presentation~$C$ from Example~\ref{ex:flange-5}.
    After pruning the third row of $C$ using the first row of $C$, which corresponds to the fact that the third generator in the flat-injective matrix in Example~\ref{ex:flange-5} is redundant, we obtain the free presentation
    \[ \begin{bmatrix}
        X_2^2 & -X_1X_2 & 0 & -X_1X_2 & X_2^2 & X_1^2 & 0 & X_1^2X_2 \\
        0 & X_1^2 & X_2 & X_1^2 & -X_1X_2 & 0 & X_1^3 & -X_1^3
    \end{bmatrix} \]
    Since this matrix does not have any unit entries, it follows that the flat-injective matrix $\tilde A$ from Example~\ref{ex:flange-1}, which has been used in Example~\ref{ex:flange-5} to compute $\partial_1$, is minimal.
    Successive elimination of the remaining columns of $\partial_1$ yields the equivalent free presentation
    \[ \begin{bmatrix}
        X_2^2 & -X_1X_2 & 0 & X_1^2 \\
        0 & X_1^2 & X_2 & 0
    \end{bmatrix} \text, \]
    which minimally generates the syzygy module of $M$.

    Because we only performed column operations and only eliminated one non-zero column with unit entry, which corresponds to the redundant generator in degree $(1, 1)$ of the free-injective matrix in Example~\ref{ex:flange-5}, we see that indeed $\tilde A$ is generator-minimal.
\end{example}

\begin{proposition}%
    \label{prop:reduced free cofree}
    Given a free-cofree presentation $\varphi\colon F_0\to E_0$ of an $n$-graded module $M$, a generator-minimal free-cofree presentation of $M$ can be constructed using relative Schreyer's presentation $\partial_1:F_1\to F_0$ and $n$-graded pruning.
\end{proposition}

\begin{proof}
    Let $\varphi \colon F_0 \to E_0$ be a free-cofree presentation. Assume that it is in Gröbner form (if it is not already in Gr\"obner form then we can always bring it in that form by relative Buchberger's algorithm). After applying monomialization, express $E_0$ as $R^d/U$ and $\im\varphi$ as $V/U\subseteq R^d/U$. Fix a POT monomial order on $R^d$. Use the relative Schreyer's presentation $\partial_1:F_1\to F_0$, and the associated Gr\"obner basis $H^{(1)}$ of the syzygy module $S\subseteq F_0$, where
    \[
    S\coloneqq\ker\varphi=\ker(\varepsilon \colon F_0 \twoheadrightarrow M)=\mathrm{im}\partial_1
    \]
    is obtained by Algorithm~\ref{alg:FreePresentation}.
    Then, by applying pruning minimization on the Schreyer's presentation of $V/U\subseteq R^d/U$, the resulting free cover
    $F\twoheadrightarrow F_0/S$
    is minimal.
    Consequently, one obtains a generator-minimal free-cofree presentation
    \[
    F \xrightarrow{f} F_0 \xrightarrow{\varphi} E_0
    \]
    of $\operatorname{im}\varphi$, where the graded matrix $f\colon F\to F_0$  is the row-reduction associated to the $n$-graded pruning of the graded matrix $\partial_1:F_1\to F_0$.

    Moreover, if we use the \emph{reduced relative Gr\"obner basis} as input, the resulting Schreyer's presentation—and hence the associated \emph{minimal free-cofree presentation}—depends only on $\operatorname{im}\varphi$ and $E$, and not on the particular choices of generating sets of $U$ and $V$.
 \end{proof}

By applying Proposition \ref{prop:reduced free cofree} and Matlis duality for the cogenerator-minimization, we obtain the following corollary.

\begin{corollary}
\label{cor:_reduced free cofree}
    Given a finitely generated and finitely supported $n$-graded $R$-submodule $M\coloneqq V/U$ of a $n$-graded injective $R$-module $E\coloneqq F/U$, there exists a minimal free-cofree which is unique for $M$ (i.e. independent of the input generators of $V$ and $U$), once a monomial order on $R$ is fixed.
\end{corollary}
In fact, even if one is given a finitely generated and finitely presented module $M$, without an explicit realization as a quotient $M = V/U$ with $U \subseteq V \subseteq F$, such a representation can be constructed. More precisely, for any finitely generated and finitely presented $n$-graded module $M$ presented as a diagram of vector spaces, one may consider its \emph{associated free--cofree presentation} as introduced in \cite{GrimpenStefanou2025}, which is determined by the action of the indeterminates $X_i$ on $M$. The resulting \emph{minimal free-cofree presentation} depends only on $M$ and the choice of a monomial order on $R$, extended to $F$ via a POT order. This yields a reduction procedure at the level of $R$-modules, allowing one to compute minimal free--cofree presentations via relative Gr\"obner bases directly from diagrammatic data of $n$-graded modules, where the multiplication maps are specified by $\Bbbk$-linear matrices.

\subsection{A linear algebra viewpoint of generator minimization}%
\label{ssect:gen-min}

In this subsection we describe a general linear algebra method for constructing nested submodules
$U \subseteq V \subseteq R^d$ together with generating sets $G$ and $H$ of $U$ and $V$, respectively, such that $V/U\cong M$ for a given Artinian graded $\Bbbk[X_1,\dotsc,X_n]$-module~$M$.
Moreover, $H$ is a reduced Gröbner basis of $V$ relative to $U$ and is in bijection with a minimal generating set of $M$.
The construction yields, e.g., the modules $V'$ and $U'$ of Example~\ref{ex:Resolution}.

Let $M = \bigoplus_{a \in \mathbb Z^n} M_a$ be a finitely generated and finitely supported $n$-graded $\Bbbk[X_1, \dotsc, X_n]$-module. Thus, $M$ is Artinian and of finite $\Bbbk$-dimension by Remark~\ref{rem:fg-fs-Artinian}. We may assume that $\supp M \subseteq \mathbb N^n$.

\begin{construction}\label{constr:ModuleConstructions}
    For each $\alpha\in\supp(M)$, we decompose $M_\alpha=O_\alpha\oplus N_\alpha$, where $O = \mathfrak m M$ (i.e.\ $O_\alpha=\sum_{i=1}^n X_i\cdot M_{\alpha-\varepsilon_i}$) and $N_\alpha$ is a choice of a complement of $O_\alpha$ in $M_\alpha$.
    For each $\alpha\in \supp(M)$ choose a basis $B_\alpha$ of $N_\alpha$.
    Set $d=\sum_{\alpha\in\supp(M)}\dim N_\alpha$.
    We will now define submodules $U\subseteq V\subseteq R^d$.

    First, note that $\mathcal{B}=\bigcup_{\alpha\in\supp(M)} B_\alpha$ is a minimal generating set of $M$. Consider the associated presentation
    \[
    \psi\colon \bigoplus_{u\in\mathcal{B}}\Sigma^{\mdeg u}R\to M\text ,
    \]
    and let $H'\coloneqq\{\mathbf{e}_u\mid u\in \mathcal{B}\}$ be the standard basis of its domain and $G'$ be the reduced Gröbner basis of $\ker(\psi)$ with respect to any position-over-term ordering. Define
    \[H\coloneqq H'^{\mon}\quad\text{and}\quad G\coloneqq G'^{\mon}\text,\]
    and let $V$ and $U$ be the submodules of $R^d$ generated by $H$ and $G$, respectively. Note that $U=\Syz(\mathcal{B})^{\mon}$.
\end{construction}

It is now straightforward to show that $M\cong V/U$ and that the sets $G$ and $H$ from Construction~\ref{constr:ModuleConstructions} are reduced (relative) Gröbner bases.

\begin{proposition}\label{prop:PropertiesOf ConstructedModules}
    Let $M=\bigoplus_{\alpha\in \mathbb{Z}^n} M_\alpha$ be a finitely supported and finite dimensional $n$-graded $R$-module with $\supp(M)\subseteq \mathbb{N}^n$.
    Consider the modules $U,V\subseteq R^d$, generated by the sets $G$ and $H$, respectively, as in Construction~\ref{constr:ModuleConstructions}.
    Then we have:
    \begin{enumerate}
        \item\label{prop:PropertiesOf ConstructedModules:1} $M\cong V/U$ as $n$-graded modules.
        \item\label{prop:PropertiesOf ConstructedModules:2} $G$ is the reduced Gr\"{o}bner basis of $U$ with respect to a position-over-term module monomial ordering on $R^d$.
        \item\label{prop:PropertiesOf ConstructedModules:3} $H$ is the reduced Gr\"{o}bner basis of $V$ relative to $U$ with respect to the compatible module monomial ordering on $R^d/U$. Moreover, $H$ is in bijection with a minimal generating set of $M$.
    \end{enumerate}
\end{proposition}
\begin{proof}
    \ref{prop:PropertiesOf ConstructedModules:1}: By the first isomorphism theorem, there is a graded isomorphism
    \[
    M\cong \left(\bigoplus_{u\in \mathcal{B}} \Sigma^{\mdeg u}R\right)/\Syz(\mathcal{B})\eqqcolon F'/U' \text .
    \]
    The claimed statement now follows from the fact that $V=F'^{\mon}$ and $U=U'^{\mon}$.

    \ref{prop:PropertiesOf ConstructedModules:2}: As $G$ is the monomialization of the position-over-term Gröbner basis $G'$ and the monomialization of an $S$-polynomial is the $S$-polynomial of the monomializations, $G$ inherits the reduced Gröbner basis structure from $G'$ with respect to the same position-over-term ordering.

    \ref{prop:PropertiesOf ConstructedModules:3}: That $H$ is a reduced Gröbner basis follows as in (2), noting that the standard basis $H'$ of Construction~\ref{constr:ModuleConstructions} is obviously the reduced Gröbner basis relative to $G'$ with respect to the same position-over-term ordering. The bijection to a minimal generating set of $M$ is given explicitly in Construction~\ref{constr:ModuleConstructions}.
\end{proof}

\section*{Acknowledgments}

Fritz Grimpen thanks Michael Kerber for fruitful discussions on the algorithmic aspects of flat-injective presentations during a research stay in March~2025 at TU Graz.

Matthias Orth was partially supported by the FWO grants G0F5921N (Odysseus) and G023721N, and by the KU Leuven grant iBOF/23/064.

Anastasios Stefanou thanks Ezra Miller, Fatemeh Mohammadi, Bernd Sturmfels, Mahrud Sayrafi, Henry Schenck,  Matías Bender, Fabian Lenzen, Michael Lesnick, and Heather Harrington for discussions on Gr\"obner bases.

\section*{CRediT authorship contribution statement}
\textbf{Fritz Grimpen: }Investigation, Methodology, Validation, Visualization, Software, Writing – original draft, Writing – review \& editing.

\noindent\textbf{Matthias Orth: }Investigation, Methodology, Validation,  Writing – original draft, Writing – review \& editing.

\noindent\textbf{Anastasios Stefanou: }Conceptualization, Methodology, Validation, Project Administration, Supervision, Writing – original draft, Writing – review \& editing.

\section*{Declaration of competing interest}
The authors declare that they have no known competing financial interests or personal relationships that could have appeared to influence the work reported in this paper.

\bibliographystyle{plainurl}
\bibliography{main}
\end{document}